\documentclass[12pt,reqno]{amsart}

\usepackage{latexsym} 
 
\usepackage[dvips]{graphics}
\usepackage{epsfig}
\usepackage{mathrsfs}
\usepackage{amssymb}
\usepackage{amsthm}
\usepackage{amsfonts}
\usepackage{amsmath}
\usepackage{amstext}
\usepackage{amscd}
\usepackage{tikz}
\usepackage{epic}
\usepackage{eepic}
\usepackage{pstricks,pst-plot}
\usepackage{enumerate}
\numberwithin{equation}{section}

\theoremstyle{plain}
\newtheorem{thm}{Theorem}[section] 
\newtheorem{prop}[thm]{Proposition}
\newtheorem{cor}[thm]{Corollary}
\newtheorem{lem}[thm]{Lemma}

\newtheorem{theorem*}{Theorem}[]

\theoremstyle{definition}
\newtheorem{defn}[thm]{Definition}

\newtheorem{example}[thm]{Example}

\newtheorem{ques}{Question}

\theoremstyle{remark}
\newtheorem{rem}[thm]{Remark}

\newcommand{\C}{\mathbb{C}}
\newcommand{\N}{\mathbb{N}}
\newcommand{\R}{\mathbb{R}}

\def\accentclass@{7}
\def\makeacc@#1#2{\def#1{\mathaccent"\accentclass@#2 }}
\makeacc@\cir{017}

\title[(SSP) Geometry]
{On the Geometry of sets satisfying  the Sequence Selection Property}

\author{Satoshi Koike and Laurentiu Paunescu}
\address{Department of Mathematics, Hyogo University of Teacher Education,
Kato, Hyogo 673-1494, Japan}
\email{koike@hyogo-u.ac.jp} 
\address{School of Mathematics and Statistics, The University of Sydney, Sydney, 
NSW, 
2006,
Australia}
\email{laurent@maths.usyd.edu.au}
\subjclass[2010]{Primary
14P15, 32B20,
Secondary
57R45}

\keywords{direction set, transversality, sequence selection property,
bi-Lipschitz homeomorphism.}

\begin{document}


\begin{abstract}

In this paper we study fundamental directional properties of sets under the 
assumption  
of condition $(SSP)$ (introduced in \cite{koikepaunescu}). We show several 
transversality 
theorems in the singular case and an $(SSP)$-structure preserving theorem.
As a geometric illustration, our transversality results are used to prove several facts 
concerning complex analytic varieties in \ref{cxvar}. Also, using our results on sets with condition (SSP), we  give a classification of spirals in the appendix \ref{spirals}.

 The $(SSP)$-property is most suitable for  understanding transversality in the 
Lipschitz category. This property is shared by a large class of sets, in particular 
by  
subanalytic sets or by  definable sets in an o-minimal structure.

\end{abstract}

\maketitle
\bigskip
\section{Introduction}

The notions of tangent cone and direction set have taken
a very important role in the study of several equisingularity type problems,
in particular, after the pioneering works of H. Whitney \cite{whitney1, whitney}, 
to the topological equisingularity problem.
For instance, I. Nakai discussed and used  directional properties
in \cite{nakai} in order to show the appearance of topological
moduli in a family of polynomial map-germs : $(\R^n,0) \to (\R^p,0)$
for $n \ge 3$, $p \ge 2$.
On the other hand, the authors showed in \cite{koikepaunescu} 
that the dimension of the common
direction set of two subanalytic subsets, called the {\em directional
dimension}, is preserved by a bi-Lipschitz homeomorphism provided that
their images are also subanalytic.  
In order to prove this result we
introduced and employed in an essential way the notion of sequence
selection property ((SSP) for short).
(SSP) is a notion based on the direction set.
It takes an important role in the study of Lipschitz equisingularity.
Using the aforementioned theorem in \cite{koikepaunescu},
we can see that the Oka family \cite{oka} is not
Lipschitz trivial as a family of zero-sets of real 
polynomial function germs.
Our aim in this paper is to study the geometry of sets
satisfying (SSP), their behaviour under  bi-Lipschitz transformations
and to point out applications 
to complex singularities and also other fields.

In order to do this,  we introduce the notions of transversality and weak
transversality, using the real cone (half-cone) of the direction set,
essential tools for understanding the sets satisfying condition (SSP).
Our main concern is to decide under which conditions the
transversality of sets is preserved by (bi-Lipschitz) homeomorphisms.
In particular we show that the transversality for complex analytic
sets is preserved by bi-Lipschitz homeomorphisms (Theorem 3.2),
provided that their images are also complex
analytic sets, and that the weak transversality for general sets is
preserved by bi-Lipschitz homeomorphisms, provided that one of them
and its image satisfy the sequence selection property (Theorems
3.5 and 3.11). 
In fact the weak transversality is preserved for arbitrary sets if the
bi-Lipschitz homeomorphism satisfies the condition semiline-(SSP),
simply a corollary of Theorem 2.25.

In addition, we introduce and study the notion of (SSP) mappings. We
show that the (SSP) structure is preserved by (SSP) bi-Lipschitz
homeomorphisms (Theorem 4.7).  In general the behaviour of a
merely bi-Lipschitz homeomorphism can be very wild in respect to the
direction sets. We show that whenever a bi-Lipschitz homeomorphism is also an (SSP)
mapping, this is no longer the case. Indeed, we are able to control this behaviour by 
either
considering it in regard to sets satisfying condition (SSP) or by
considering bi-Lipschitz homeomorphisms endowed with extra properties.
In particular, we look for those homeomorphisms
with a good directional behaviour and we single out two large classes of
examples.

\section{Directional Properties of Sets}\label{}
\medskip

Let us recall our notion of  direction set. For simplicity in this paper we only 
consider  the direction sets  at the origin.

\begin{defn}\label{directionset}
Let $A$ be a set-germ at $0 \in \R^n$ such that
$0 \in \overline{A}$.
We define the {\em direction set} $D(A)$ of $A$ at $0 \in \R^n$ by
$$
D(A) := \{a \in S^{n-1} \ | \
\exists  \{ x_i \} \subset A \setminus \{ 0 \} ,
\ x_i \to 0 \in \R^n  \ \text{s.t.} \
{x_i \over \| x_i \| } \to a, \ i \to \infty \}.
$$
Here $S^{n-1}$ denotes the unit sphere centred at $0 \in \R^n$.
\end{defn}

For a subset $A \subset S^{n-1}$, we denote by $L(A)$
a half-cone of $A$ with the origin $0 \in \R^n$ as the vertex:
$$
L(A) := \{ t a \in \R^n\ | \ a \in A, \ t \ge 0 \}.
$$
In the case $A$ is a point (not the origin) we call $L(A)$ a {\it semiline}. 
For a set-germ $A$ at $0 \in \R^n$ such that $0 \in \overline{A}$, 
we put $LD(A) := L(D(A))$, and call it the {\em real tangent cone}
at $0 \in \R^n$.

Let $U, \ V \subset \R^n$ such that 
$0\in \overline{U} \cap \overline{V}$. 
The following are true:
\begin{enumerate} 
\item $D(\overline{U})=D(U)$
\item $D(U\cup V)=D(U)\cup D(V)$
\item $\overline{\cup_iD(U_i)}\subseteq D(\cup U_i)$
\item If $U_i$ are half-cones then $\overline{\cup_iD(U_i)}= D(\cup U_i)$
\item  $D(U\cap V)\subseteq  D(U)\cap D(V)$
\end{enumerate}



\subsection{Condition (SSP)}
\label{condition(SSP)}

In \cite{koikepaunescu} sea-tangle properties and directional
properties of sets with the sequence selection property played an
essential role in the proof of the main theorem (cf. Theorem \ref{main
  theorem}).  For the reader's convenience let us recall the main
theorem in \cite{koikepaunescu}.  See H. Hironaka \cite{hironaka} for
the definition of subanalyticity.

\begin{thm}\label{main theorem}
  (Main Theorem in \cite{koikepaunescu}) Let $A$, $B \subset \R^n$ be
  subanalytic set-germs at $0 \in \R^n$ such that $0 \in \overline{A}
  \cap \overline{B}$, and let $h : (\R^n,0) \to (\R^n,0)$ be a
  bi-Lipschitz homeomorphism.  Suppose that $h(A), \ h(B)$ are also
  subanalytic.  Then we have the equality of dimensions,
$$
\dim (D(h(A)) \cap D(h(B))) = \dim (D(A) \cap D(B)).
$$
\end{thm}

We denote by $(SSP)$ the sequence selection property for short.
Here we introduce a generalised notion of (SSP) 
relatively to a subset of $\R^n$.

\begin{defn}\label{SSP}
Let $A, B$ be two  set-germs at $0 \in \R^n$
such that $0 \in \overline{A} \cap \overline{B}, D(A)\subseteq D(B) $.
We say that $A$ satisfies {\em condition} $(SSP)${\it-relative to} $B$,
if for any sequence of points $\{ a_m \}$ of $B$
tending to $0 \in \R^n$ such that 
$\lim_{m \to \infty} {a_m \over \| a_m \| } \in D(A)$,
there is a sequence of points $\{ b_m \} \subset A$ such that
$$
\| a_m - b_m \| \ll \| a_m \|, \ \| b_m \| . 
$$ 

In the case $B= \R^n$ we will not mention $B$ (it is the usual (SSP) condition).
\end{defn}

Concerning this relative condition $(SSP)$,
we can easily show the following:

\begin{prop}\label{transitivity}
The relative  condition $(SSP)$ 
is transitive, namely if $A$ satisfies condition $(SSP)$-relative 
to $B$ and $B$ satisfies condition $(SSP)$-relative to $C$, 
then $A$ satisfies condition $(SSP)$-relative to $C$.
\end{prop}








We give  some remarks on the relative condition $(SSP)$
( (2) and (3) follow from the above proposition).

\begin{rem}

\noindent
\begin{enumerate}
\item $A$ (resp. $\overline A$) satisfies  condition $(SSP)$-relative 
to $\overline A$ (resp. $A$).

\item $A$ satisfies condition $(SSP)$ if and only if  
$A$ satisfies  condition $(SSP)$-relative to $LD(A)$.

\item
 
$A$ satisfies  condition $(SSP)$ if and only if  
$\overline A$ satisfies  condition $(SSP).$
\item Let $A \subset \R^n$ be a set-germ at $0\in \R^n$ such that $0 \in \overline{A}$,
and let $d, \ C > 0$.
The {\em sea-tangle neighbourhood $ST_d(A;C)$ of $A$,
of degree $d$ and  width} $C$, is defined by:
$$
ST_d(A;C) := \{ x \in \R^n \ | \ dist (x,A) \le C |x|^d \}.
$$

Then,  $A$ satisfies condition $(SSP)$-relative to $ST_d(A;C), d>1.$ 
\end{enumerate}
\end{rem}

In this paper we consider also the notion of weak
sequence selection property, denoted by $(WSSP)$ for short.

\begin{defn}\label{WSSP}
Let $A, B$ be two set-germs at $0 \in \R^n$
such that $0 \in \overline{A} \cap \overline{B}, D(A)\subseteq D(B) $.
We say that $A$ satisfies {\em condition} $(WSSP)${\it-relative to} $B$,
if for any sequence of points $\{ a_m \}$ of $B$
tending to $0 \in \R^n$ such that 
$\lim_{m \to \infty} {a_m \over \| a_m \| } \in D(A)$,
there is a subsequence $\{ m_j \}$ of $\{ m \}$ with 
$\{ b_{m_j} \} \subset A$ such that
$$
\| a_{m_j} - b_{m_j} \| \ll \| a_{m_j} \|, \ \| b_{m_j} \| . 
$$ 
\end{defn}

We have the following characterisation of condition $(SSP)$.
The proof  in the relative case is similar to
the non-relative case  for  which we gave a detailed proof 
in \cite{koikeloipaunescushiota}.
We sketch  a slightly rough proof here.

\begin{prop}\label{SSPWSSP}
Let $A, B $ be two set-germs at $0 \in \R^n$
such that $0 \in \overline{A}\cap \overline{B}$.
If $A$ satisfies condition $(WSSP)$-relative to $B$, 
then it satisfies condition $(SSP)$-relative to $B$.
Namely, the conditions relative $(SSP)$ and relative $(WSSP)$ are equivalent.
\end{prop}

\begin{proof}
Assume that $A$ does not satisfy condition $(SSP)$.
Then there is a sequence of points $\{ a_m \in B\}$
tending to $0 \in \R^n$ such that
$\lim_{m \to \infty} {a_m \over \| a_m \|} \in D(A)$
and $\lim_{m \to \infty} {d(a_m,A) \over \| a_m \|} = \alpha > 0$,
where $d(a_m,A)$ denotes the distance between $a_m$ and $A$.
This implies that there does not exist a sequence of points
$\{ b_m \} \subset A$ such that $\| a_m - b_m \| \ll \| a_m \|$.
Therefore $A$ does not satisfy condition $(WSSP)$.
\end{proof}

We make  some remarks on $(SSP)$: 

\begin{rem}\label{remark1} 
\noindent

\begin{enumerate}

\item  

In fact one can easily see that  $A$ satisfies condition 
$(SSP)$-relative 
to $B$ if and only if  for any sequence of points $\{ a_m \}$ of $B$
tending to $0 \in \R^n$ such that 
$\lim_{m \to \infty} {a_m \over \| a_m \| } \in D(A)$, then  
$ d(a_m,A)\over \|a_m\|$ tends to $0 \in \R$. (Or there is a subsequence 
which tends to zero.)

\item  

Condition $(SSP)$ is $C^1$ invariant,
but not bi-Lipschitz invariant (cf. \S 5 in \cite{koikepaunescu}).
Note that condition $(SSP)$ is invariant
under a bi-Lipschitz homeomorphism $h : (\R,0) \to (\R,0)$.
We leave the proof of this fact to the interested reader.
\end{enumerate}
\end{rem} 

As stated in the above remark,  the condition $(SSP)$ is not
bi-Lipschitz invariant.
However if a map $h$ is bi-Lipschitz, we have the following: 

\begin{lem}\label{presdir}
Let $h : (\R^n,0) \to (\R^n,0)$ be a bi-Lipschitz homeomorphism,
and let $A, B $ be two set-germs at $0 \in \R^n$
such that $0 \in \overline{A}\cap \overline{B}$.
Then $A$ satisfies condition $(SSP)$-relative 
to $B$ and $D(B)= D(A)$ if and only if $h(A)$ satisfies
condition $(SSP)$-relative 
to $h(B)$ and $D(h(B))= D(h(A))$. 
From this we can conclude that if $A$ satisfies condition 
$(SSP)$, then $Dh(A)=Dh(LD(A)))$ and $h(A)$ satisfies condition 
$(SSP)$-relative to $h(LD(A))$ ($B=LD(A)$).
\end{lem}

\begin{proof} Use (1) of remark \ref{remark1}.
\end{proof}
Below we give several examples of sets satisfying the condition (SSP).

\begin{rem}\label{remark2}
Let $A, B \subseteq \R^n$ be  set-germs at $0 \in \R^n$
such that $0 \in \overline{A}\cap \overline{B}$.
\begin{enumerate}
\item The cone $LD(A)$ satisfies condition $(SSP)$.

\item If $A$ is subanalytic or definable in an o-minimal structure, then it satisfies condition $(SSP)$.

\item If $A$ is a finite union of sets, all of which satisfy 
condition $(SSP)$, then $A$ satisfies condition $(SSP)$.

\item If $0\in A$, a $C^1$ manifold, then it satisfies condition 
$(SSP)$ and $LD(A)=T_0(A)$. 
(This is not necessarily true for $C^0$ manifolds or if 
$0\notin A$.)

\item If $A\subseteq B, D(A)=D(B), A$ satisfies condition  $(SSP)$, then $B$ 
satisfies 
condition  $(SSP)$.

\item  If $A\cup \{0\}$ is path connected with $D(A)$ a point, then $A$ satisfies 
condition  $(SSP)$. The trajectories of the gradient flow of an analytic function 
satisfy 
this property; this is the famous gradient conjecture of R. Thom, proven in 
\cite{KMP}. 
They may not be always subanalytic.  

\item  If $D(A)=\{a_1,...,a_k\}$  and there are  subsets 
$A_i\subseteq A,  D(A_i)=\{a_i\}$  and  $A_i\cup\{0\}, i=1,...,k,$ are 
path connected, 
then $A$ satisfies condition  $(SSP)$.
\end{enumerate}
\end{rem}

We give one more important example satisfying condition $(SSP)$.

\begin{prop}\label{h(LD(A))}(Proposition 6.3 in \cite{koikepaunescu})
Let $h : (\R^n,0) \to (\R^n,0)$ be a bi-Lipschitz homeomorphism,
and let $A$, $h(A) \subset \R^n$ be 
subanalytic set-germs at $0 \in \R^n$ such that
$0\in \overline{A}$.
Then the set $h(LD(A))$ satisfies condition $(SSP)$.
\end{prop}

Concerning the condition (SSP) it is important to remember  that $LD(A)$ satisfies  
condition (SSP) for any subset $A, 0\in \overline A$. Accordingly we will try to 
replace $A$ by its real tangent  cone $LD(A)$ whenever possible and convenient. The 
remaining results of this subsection are in this spirit.  We recall the following  
lemma.

\begin{lem}\label{directionrel}(Lemma 5.6 in \cite{koikepaunescu})
Let $h : (\R^n,0) \to (\R^n,0)$ be a bi-Lipschitz homeomorphism,
and  let $A\subset \R^n$ be 
a set-germ at $0\in \R^n$ such that $0\in \overline A$.
Then $D(h(A)) \subset D(h(LD(A))).$
If $A$ satisfies condition $(SSP)$ or if $h$ is a $C^1-$diffeomorphism 
the equality holds.
\end{lem}

Using the above lemmas we can improve Proposition \ref{h(LD(A))}.
In fact, we gave an improvement in the non-relative case
in \cite{koikeloipaunescushiota}.
Here we generalise it to the relative case.  

\begin{thm}\label{2SSP}
Let $h : (\R^n,0) \to (\R^n,0)$ be a bi-Lipschitz homeomorphism, and let
$A \subset \R^n$ be a set-germ at $0\in \R^n$
such that $0\in \overline{A}$,  and  $B \subset \R^n$ a set-germ at $0\in \R^n$ 
such that $0\in \overline{B}$.
Assume that  $A$ satisfies condition $(SSP)$. 
Then  $h(A)$ satisfies condition $(SSP)$-relative to $B$ 
if and only if  $h(LD(A))$ satisfies condition $(SSP)$-relative to $B$. 
\end{thm}

\begin{proof}
Let us assume that $h(A)$ satisfies condition $(SSP)$-relative 
to $B$.
By assumption, $A$ satisfies condition $(SSP)$.
Therefore it follows from Lemma \ref{directionrel} that
$D(h(LD(A))) = D(h(A))$.
Let $\{ y_m \}$ be an arbitrary sequence of points of $B$
tending to $0 \in \R^n$ such that
$$
\lim_{m \to \infty} {y_m \over \| y_m \|} \in D(h(LD(A))) = D(h(A)).
$$
Let $y_m = h(x_m)$ for each $m$.
Since $h(A)$ satisfies condition $(SSP)$-relative $B$, 
there is a sequence of points
$\{ z_m \} \subset A$ such that
$$
\| h(x_m) - h(z_m) \| \ll \| h(x_m) \| , \ \| h(z_m) \| .
$$

On the other hand, there is a subsequence $\{ z_{m_j} \}$
of $\{ z_m \}$ such that 
$\lim_{{m_j} \to \infty} {z_{m_j} \over \| z_{m_j} \|} \in D(A)$.
Since $LD(A)$ satisfies condition $(SSP)$, there is a sequence 
of points $\{ \theta_{m_j} \} \subset LD(A)$ such that
$$
\| z_{m_j} - \theta_{m_j} \| \ll \| z_{m_j} \| , \ \| \theta_{m_j} \| .
$$
It follows from  $h$ being bi-Lipschitz  that
$$
\| h(z_{m_j}) - h(\theta_{m_j}) \| \ll \| h(z_{m_j}) \| , \ 
\| h(\theta_{m_j}) \| .
$$
Then we have
$$
\| h(x_{m_j}) - h(\theta_{m_j}) \| \le
\| h(x_{m_j}) - h(z_{m_j}) \| + \| h(z_{m_j}) - h(\theta_{m_j}) \| 
\ll \| h(z_{m_j}) \| .
$$
Therefore we have
$$
\| h(x_{m_j}) - h(\theta_{m_j}) \| \ll \| h(x_{m_j}) \| , \ 
\| h(\theta_{m_j}) \| .
$$
Thus $h(LD(A))$ satisfies condition $(WSSP)$-relative to $B$, 
and also condition $(SSP)$-relative to $B$ by Proposition \ref{SSPWSSP}. 
The other claim can be proved in a similar way.
\end{proof}

Note that even if both $h(A)$ and $ h(LD(A))$  satisfy condition $(SSP)$, 
it does not imply that $A$ satisfies condition $(SSP)$ 
(the spiral example, Figure 1 below).

\begin{prop}\label{conicclosure}
Let $h : (\R^n,0) \to (\R^n,0)$ be a bi-Lipschitz homeomorphism,
and let $A \subset \R^n$ be a set-germ at $0\in \R^n$ such that $0\in \overline{A}$.
Then $LD(h(A))=LD(h(LD(A)))$ and $h(LD(A))$ satisfy condition $(SSP)$ 
if and only if $LD(h^{-1}(LD(h(A))))=LD(A)$ and 
$h^{-1}(LD(h(A)))$ satisfy condition $(SSP)$.
\end{prop}

\begin{proof}
As our conditions are symmetric in $h$ (our bi-Lipschitz homeomorphism) it 
suffices to prove only the ``if'' part implication.
Since $h^{-1}(LD(h(A)))$ satisfies condition $(SSP)$ it follows that 
$$LD(h(A))=LD(h(h^{-1}(LD(h(A)))))=LD(h(LD(h^{-1}(LD(h(A)))))),$$ 
and because we always have
$$LD(h(LD(A)))=LD(h(LD(h^{-1}(h(A)))))\subseteq LD(h(LD(h^{-1}(LD(h(A))))))$$ 
it follows that $LD(h(A))=LD(h(LD(A)))$.

Assume that  $\{ h(y_m) \}$ is an arbitrary sequence of points of $\R^n$
tending to $0 \in \R^n$ such that
$$
\lim_{m \to \infty} {h(y_m) \over \|h( y_m) \|} \in D(h(LD(A))) = D(h(A)).
$$
As cones satisfy condition $(SSP)$ we can assume that 
$ h(y_m) \in LD(h(LD(A)))=LD(h(A))$,
so $y_m \in h^{-1}(LD(h(A)))$. Passing to a subsequence, 
if necessary, we may assume that in fact 
$\lim_{m \to \infty} {y_m \over \| y_m \|} \in  D(h^{-1}(LD(h(A))))=LD(A)$. 
Again as cones satisfy  condition $(SSP)$ we can claim 
the existence of a sequence
$x_i \in LD(A)$ such that  
$$
\| y_i - x_i \| \ll \| x_i \| , \ \| y_i \| .
$$
The fact that  $h$ is bi-Lipschitz implies that
$$
\| h(x_i) - h(y_i) \| \ll \| h(x_i) \| , \ 
\| h(y_i) \| .
$$
As $h(x_i)\in h(LD(A))$ we proved that $h(LD(A))$ satisfies 
condition $(SSP)$.
\end{proof}

\begin{rem}\label{slowspiral}
In order to show $h(LD(A))$ satisfies condition $(SSP)$, 
we cannot drop the assumption $LD(h^{-1}(LD(h(A))))=LD(A)$. 
Indeed if $h$ is the spiral bi-Lipschitz homeomorphism of Example 3.3 
in \cite{koikepaunescu}, we put $A=\R \times 0$ so that 
$h(LD(A))=h(A)$ is a spiral which does not satisfy condition
$(SSP)$ ( Figure 1 above). 
Clearly $h^{-1}(LD(h(A)))=\R^2$ so it satisfies condition $(SSP)$, 
and $LD(h^{-1}(LD(h(A))))=\R^2 \neq LD(A)$. 
\end{rem}

\begin{figure}
  \begin{tikzpicture}[scale=.5]
    \draw[->] (-7,0) -- (7,0) node[below] {$x$};
    \draw[->] (0,-4) -- (0,8) node[left] {$y$};
    \clip (-7,-4) rectangle (7,7.5);
    \draw plot[samples=1000,parametric,domain=-30:10,id=log-spiral]
    function {exp(.25*t)*cos(t),exp(.25*t)*sin(t)};
  \end{tikzpicture}
\caption{}
\label{logspiral}
\end{figure}


In  the same spirit we have the following.

\begin{prop}Let $h : (\R^n,0) \to (\R^n,0)$ be a bi-Lipschitz homeomorphism,
and let $A \subset \R^n$ be a set-germ at $0\in \R^n$ such that $0\in \overline{A}$.
The following are equivalent:
\begin{enumerate}
\item  $A, h(A)$ both satisfy condition $(SSP)$.
\item  $A, h^{-1}(LD(h(A)))$  both satisfy condition $(SSP)$ and 
$LD(h^{-1}(LD(h(A))))=LD(A)$.
\item $h(A), h(LD(A))$  both satisfy condition $(SSP)$ and 
$LD(h(LD(A)))=LD(h(A))$.
\end{enumerate}
\end{prop}

\begin{example} For instance, the situation in the above result happens  in the 
following two general cases.

\begin{enumerate}

\item If both $A, h(A)$ are subanalytic or definable in an $o$-minimal structure over $\R$,

\item If $A$ satisfies condition $(SSP)$ and $h$ 
is a $C^1-$diffeomorphism.
\end{enumerate}
\end{example}


\subsection{Condition semiline-(SSP)}\label{lSSP}
Our general purpose is to provide a large class of examples of homeomorphisms which 
preserve the condition (SSP).
In this subsection we introduce the condition
semiline-$(SSP)$, and we use it to give some characterisations
of the condition $(SSP)$.
In particular, in the bi-Lipschitz case, we prove that the condition 
semiline-(SSP) is equivalent to 
preserving 
the condition (SSP) (Corollary \ref{linecor2}). Furthermore we prove that a 
semiline-(SSP) bi-Lipschitz homeomorphism $h$, induces a ``positive homogeneous'' 
bi-Lipschitz homeomorphism which corresponds the real cones of arbitrary sets $A$ 
and their 
images $h(A)$ (Theorem \ref{samedimension}).
\begin {defn}
We say that a homeomorphism $h : (\R^n,0) \to (\R^n,0)$ satisfies
condition {\em semiline}-(SSP), if $h(\ell)$ has a unique direction
for all semilines $\ell$.
\end{defn}

\begin{rem}
The bi-Lipschitz homeomorphisms satisfying condition semiline-(SSP)
are bi-Lipschitz homeomorphisms which are  Gateaux right differentiable. 
\end{rem}
\begin{prop}\label{semiline}
  Let $h : (\R^n,0) \to (\R^n,0)$  be a bi-Lipschitz 
homeomorphism.
Suppose that $h^{-1}(\tau)$ satisfies condition (SSP) for all semilines $\tau$.
Then $LD(h(\ell))$ is a semiline for all semilines $\ell$, that is $h$ satisfies 
condition
{\em semiline}-(SSP). (In particular $h(\ell)$ satisfies condition (SSP).)

\end{prop}
\begin{proof}
Indeed, take a semiline $\ell$ and sequences of points
$\{ b_i\}$, $\{ c_i\} \subset \ell$ tending to $0 \in \R^n$ such that 
$LD(h(\{ b_i \} )) = \ell_1$ and  $LD(h(\{ c_i \} )) = \ell_2$,
where $\ell_1$, $\ell_2$ are semilines.
Since $\ell_1$ (resp. $\ell_2$) satisfies condition (SSP), there is
a sequence of points $\{ b_i^{\prime}\}$ with $\{ h(b_i^{\prime})\}
\subset \ell_1$ (resp. $\{ c_i^{\prime}\}$ with $\{ h(c_i^{\prime})\}
\subset \ell_2$) such that
$$
\| h(b_i ) - h(b_i^{\prime}) \| \ll \| h(b_i )\|, \ \| h(b_i^{\prime})\|
\ \ (\text{resp.} \ \| h(c_i ) - h(c_i^{\prime}) \| \ll \| h(c_i )\|, 
\ \| h(c_i^{\prime})\|).
$$ 
It follows that
\begin{equation}\label{sspineq1}
\| b_i - b_i^{\prime} \| \ll \| b_i \|, \ \| b_i^{\prime}\| \ \
(\text{resp.} \ \| c_i - c_i^{\prime}\| \ll \| c_i\|, \ \| c_i^{\prime}\|).
\end{equation}

On the other hand, we have
$$
\{\lim_{i \to \infty} {b_i \over \| b_i \|}\} =
\{\lim_{i \to \infty} {c_i \over \| c_i \|}\} = D(\ell )
$$
and $\{ c_i^{\prime}\} \subset h^{-1}(\ell_2)$. 
By (\ref{sspineq1}), we have
$$
\{\lim_{i \to \infty} {b_i^{\prime} \over \| b_i^{\prime} \|}\} =
\{\lim_{i \to \infty} {b_i \over \| b_i \|}\} = D(\ell )
\subset  D(h^{-1}(\ell_2)). 
$$
Since $h^{-1}(\ell_2)$ satisfies condition (SSP), there is a sequence
of points $\{ c_i''\}$ with $h(\{ c_i''\}) \subset \ell_2$ such that 
$$
\| b_i^{\prime} - c_i''\| \ll \| b_i^{\prime}\|, \ \| c_i''\|.
$$ 
This implies that
$$
D(\ell_1 ) = 
\{\lim_{i \to \infty} {h(b_i^{\prime}) \over \| h(b_i^{\prime})\|}\}
= \{\lim_{i \to \infty} {h(c_i'') \over \| h(c_i'') \|}\} = D(\ell_2),
$$
that is $\ell_1=\ell_2$.
\end{proof}

We have the following corollaries.

\begin{cor}
In the case of a bi-Lipschitz homeomorphism, the condition semiline-(SSP) is 
equivalent 
with asking that $h(\ell)$ satisfies condition (SSP) for all semilines $\ell$. 
Moreover 
in the bi-Lipschitz  case it follows that $h$ satisfies condition semiline-(SSP) is 
equivalent to $h^{-1}$ satisfies condition semiline-(SSP).
\end{cor}
\begin{proof}

 Indeed assume that $h(\ell)$ satisfies condition (SSP) for all semilines $\ell$. 
From 
the result above it follows that $h^{-1}$  satisfies  condition semiline-(SSP), and 
therefore it satisfies condition (SSP) as well. This  in turn shows that $h$ 
satisfies 
condition  semiline-(SSP). 
 
 \end{proof}

\begin{cor}\label{linecor1}
Let $h : (\R^n,0) \to (\R^n,0)$  be a bi-Lipschitz homeomorphism,
and let $A \subset \R^n$ be a set-germ at $0\in \R^n$ such that 
$0 \in \overline{A}$.
Suppose that $A$ satisfies condition (SSP), and $h$ satisfies condition 
semiline-(SSP).
Then $h(A)$ satisfies condition (SSP).
\end{cor}

\begin{proof}
Let $\ell$ be an arbitrary semiline contained in $LD(h(A))$.
Then there is a sequence of points $\{ a_i \} \subset A$
tending to $0 \in \R^n$ such that $LD(\{ h(a_i)\}) = \ell$.
Since $\ell$ satisfies condition (SSP), there is a sequence of points
$\{ c_i \}$ with $\{ h(c_i)\} \subset \ell$ such that
$$
\| h(a_i) - h(c_i)\| \ll \| h(a_i)\|, \ \|h(c_i)\|.
$$
It follows that 
$$
\| a_i - c_i\| \ll \| a_i\|, \ \| c_i\|.
$$
Therefore we have $LD(\{ a_i\} ) = LD(\{ c_i\}) \subset LD(h^{-1}(\ell))$.
We can use the previous proposition to claim that
$LD(\{ a_i\} ) = LD(\{ c_i\}) = LD(h^{-1}(\ell))$ is a semiline 
$\ell_1 \subset LD(A)$. 

Let $\{ b_i\}$ be an arbitrary sequence of points tending to
$0 \in \R^n$ such that $LD(\{ h(b_i)\} ) = \ell \subset LD(h(A))$.
Since $\ell$ satisfies condition (SSP), there is a sequence
of points $\{ b_i^{\prime}\}$ with 
$\{ h(b_i^{\prime})\} \subset \ell$ such that
$$
\| h(b_i) - h(b_i^{\prime})\| \ll \| h(b_i)\| , \
\| h(b_i^{\prime})\|.
$$
It follows that
\begin{equation}\label{sspineq2}
\| b_i - b_i^{\prime}\| \ll \| b_i\| , \ \| b_i^{\prime}\|.
\end{equation}
Note that $\{ b_i^{\prime} \} \subset h^{-1}(\ell )$.
Therefore we have
$$
LD(\{ b_i^{\prime}\} ) = LD(h^{-1}(\ell )) = \ell_1 \subset LD(A).
$$
Since $A$ satisfies condition (SSP), there is a sequence of points
$\{ b_i'' \} \subset A$ such that
\begin{equation}\label{sspineq3}
\| b_i^{\prime} - b_i''\| \ll \| b_i^{\prime}\| , \ \| b_i''\| .
\end{equation}
By (\ref{sspineq2}) and (\ref{sspineq3}), we have
$$
\| b_i - b_i''\| \ll \| b_i\| , \ \| b_i''\|.
$$
It follows that
$$
\| h(b_i) - h(b_i'')\| \ll \| h(b_i)\| , \ \| h(b_i'')\|.
$$
Since $\{ h(b_i'')\} \subset h(A)$, $h(A)$ satisfies
condition (SSP).
\end{proof}

Using the above corollary, we can see the following:

\begin{cor}\label{linecor2}
Let $h : (\R^n,0) \to (\R^n,0)$  be a bi-Lipschitz 
homeomorphism.
Then the following are equivalent:

(1) $h$ has the property that for any set-germ at $0\in \R^n, A \subset \R^n$
such that $0 \in 
\overline A,$ we have that $ A$ satisfies condition (SSP) if and only if h(A) 
satisfies 
condition (SSP).

(2) $h$ ( so $h^{-1}$ ) satisfies condition semiline-(SSP).
\end{cor}

\begin{rem}\label{inducedmap}
Take a germ of a semiarc  $\gamma : ([0,\epsilon),0) \to (\R^n,0)$ 
with a unique direction, say $\ell = LD(\gamma)$. 
(It is not difficult to see that $\gamma$ satisfies condition (SSP).) 
It follows from Proposition \ref{semiline} that for a bi-Lipschitz
homeomorphism $h : (\R^n,0) \to (\R^n,0)$ where $h^{-1}$ satisfies condition 
semiline-(SSP), we do have that $h(\gamma)$ has also a unique direction. 
Indeed, we can easily see that $LD(h(\gamma))=LD(h(LD(\gamma)))=LD(h(\ell ))$ is 
also a semiline. 
Let
$$
\mathscr{SL} := \{\gamma :([0,\epsilon),0) \to (\R^n,0) \ | \
LD(\gamma) \ \text{is a semiline}\}.
$$
The above argument implies that if $h^{-1}$ satisfies condition
semiline-(SSP), then the map $h : \mathscr{SL} \to \mathscr{SL}$
induces a map $\overline{h} : S^{n-1} \to S^{n-1}$ defined by
$\overline{h}(D(\gamma )) = D(h(\gamma ))$ for
$\gamma \in \mathscr{SL}$. 
If both $h, h^{-1}$ satisfy condition semiline-(SSP),
then $\overline{h} : S^{n-1} \to S^{n-1}$ is a one-to-one correspondence,
in other words, $\overline{h} : S^{n-1} \to S^{n-1}$ is bijective. 

Note that in the case where $\gamma :([0,\epsilon),0) \to (\C^n,0), \
\gamma \in \mathscr{SL}$, we have that the complex tangent cone, $LD^*(\gamma)$,
 is a complex line, 
and all complex lines can be obtained in this way (see \ref{CSSP} for a definition of the complex tangent cone).

\end{rem}

\begin{thm}\label{samedimension}
Let $h: (\R^n,0) \to (\R^n,0)$ be a bi-Lipschitz homeomorphism
such that $h$ ( so $h^{-1}$)  satisfies condition semiline-(SSP). 
Then the induced map $\overline{h} : S^{n-1} \to S^{n-1}$ given
in Remark \ref{inducedmap}  extends to  a bi-Lipschitz homeomorphism  
$\overline h:\R^n\to \R^n$, and for any set-germ at $0\in \R^n, A \subset \R^n$
  such that $0 \in \overline A,$ we have 
$\overline{h}(D(A))=D(h(LD(A)))=D(h(A))$.
In particular we have $\dim\, D(A) = \dim\, D(h(A)).$
\end{thm}

\begin{proof} First we  prove the result for $A$ which satisfies condition  (SSP). 
Indeed $D(A)=D(LD(A))$ and the latest satisfies condition (SSP).
 
Let us put $\ell_a := \{ta \ | \ t \geq 0\}$ 
for $a \in S^{n-1}$. 
Then we have $LD(A) = \cup_{a\in D(A)} \ell_a$.
Let us assume that $A$ satisfies condition (SSP), 
then we have the following:

\begin{eqnarray*}
\overline{h}(D(A))
&=& (\cup_{a\in D(A)} LD(h(\ell_a))) \cap S^{n-1} \\
&\subseteq& LD(\cup_{a\in D(A)} h(\ell_a)) \cap S^{n-1} \\
&=& LD(h(\cup_{a\in D(A)} \ell_a)) \cap S^{n-1} \\
&=& LD(h(LD(A))) \cap S^{n-1} = D(h(A)).
\end{eqnarray*}

By Corollary \ref{linecor1}, $h(A)$ also satisfies condition (SSP).
Using the same argument as above, we have 
$$
(\overline{h})^{-1}(D(h(A))) \subset D(h^{-1}(h(A))) = D(A).
$$ 
It follows that 
$$
D(h(A)) \subset \overline{h}(D(A)) \subset D(h(A)).
$$
Therefore we have $\overline{h}(D(A)) = D(h(A))$.

Since $h : (\R^n,0) \to (\R^n,0)$ is a bi-Lipschitz homeomorphism,
there are positive numbers $K_1, \ K_2 \in \R$
with $0 < K_1 \le K_2$, called {\em Lipschitz constants}, such that
$$
K_1 \| x_1 - x_2 \| \le \| h (x_1) - h (x_2) \|
\le K_2 \| x_1 - x_2 \|
$$
in a small neighbourhood of $0 \in \R^n$.
Let $\overline{h} : S^{n-1} \to S^{n-1}$, $S^{n-1} \subset \R^n$,
be the mapping defined by
$$
\overline{h}(a) = \lim_{t \to 0} {h(ta) \over \| h(ta) \|}.
$$
Let $a, b \in S^{n-1}$.
Then for sufficiently small, arbitrary $t > 0$, we have
\[\begin{array}{llll}
\| {h(ta) \over \| h(ta)\|} - {h(tb) \over \| h(tb)\|} \|
&\leq {\|h(ta) - h(tb)\| \over \min(\|h(ta)\|,\|h(tb)\|)}
 \end{array}
\]
\[\begin{array}{llll}
 {\|h(ta) - h(tb)\| \over \min(\|h(ta)\|,\|h(tb)\|)}
&\leq {K_2 \| ta - tb\| \over \min(K_1 \|ta\|,K_1 \|tb\|)}
\leq {K_2 \over K_1} \| a - b\| .
 \end{array}
\]

\medskip
Taking the limit as $t \to 0^+$,
we have $\| \bar{h}(a)-\bar{h}(b)\| \leq {K_2 \over K_1}
\| a - b \|$.
Therefore it follows that $\bar{h}$ is a  bi-Lipschitz
homeomorphism.
It is not difficult to extend $\overline h$ to a global bi-Lipschitz homeomorphism, 
we 
put $\overline h(tx)
=t\overline h(x), \, x\in S^{n-1}$ (its radial extension).

Our proof  shows that in fact  $D(h(A))\subseteq D(h(LDA))=\overline{h}(D(A))$  for 
any 
$A$. Because 
$\overline {h}^{-1}=\overline{h^{-1}}$,  the equality 
$D(h(A))=D(h(LDA))=\overline{h}(D(A))$ holds in general.

\end{proof}
\begin{rem}
In particular the above property holds for any definable bi-Lipschitz 
homeomorphism, 
and 
for any subanalytic bi-Lipschitz homeomorphism (for the subanalytic case see
\cite {andreas}).
\end{rem}
\begin{rem}The assumption on $h$ cannot be much relaxed. 
Indeed,  consider a bi-Lipschitz zig-zag  homeomorphism $h: \R \to \R$ 
(in particular it preserves the $(SSP)$ property) whose graph is like 
in example \ref{example31}, Figure 3 below. 
Then $H := 1 \times h :\R \times \R \to \R \times \R$ is a bi-Lipschitz 
homeomorphism and for the semiline $ \ A = \{(t,t) \ | \ t\geq 0\}, \, H(A)$ is 
exactly 
that part of the graph of $h$ which is a zigzag.
Therefore $\dim D(H(A)) = 1$ (even $A$ satisfies condition (SSP)),
but $D(A)$ is only a point. Clearly $H$ does not satisfy semiline-$(SSP)$.
\end{rem}

\begin{cor}
Let $h: (\R^n,0) \to (\R^n,0)$ be a bi-Lipschitz homeomorphism
such that $h$ ($ h^{-1}$)  satisfies condition semiline-(SSP), and let 
$A\subset \R^n$ be a set-germ at $0\in \R^n$
such that $0\in \overline A$. 
Then $LD(A)$ and $LD(h(A))$ are bi-Lipschitz homeomorphic.

\end{cor}
\begin{proof}
Indeed by the previous result we have that $D(A)$ and $D(h(A))$ are bi-Lipschitz 
homeomorphic, and the radial extension of $\overline h$ gives the result.

\end{proof}

\medskip


\subsection{Directional properties of intersection sets}

In this subsection we treat some directional
properties of intersections.
 Even if $A$, $B$ satisfy condition $(SSP)$,
$A \cap B$ does not always satisfy condition $(SSP)$.

\begin{prop}\label{closed cone}
Let $h : (\R^n,0) \to (\R^n,0)$ be a bi-Lipschitz homeomorphism,
and let $U$, $V \subset \R^n$ be closed cones with $0 \in \R^n$
as the vertex.
Suppose that $h(U)$ satisfies condition $(SSP)$.
Then $D(h(U \cap V)) = D(h(U)) \cap D(h(V))$.
\end{prop}

\begin{proof}
Since the inclusion $\subset$ is obvious,
we show $\supset$ here.
Let $\alpha$ be an arbitrary element of $D(h(U)) \cap D(h(V))$.
Then there is a sequence of points $\{ a_m \} \subset V$
tending to $0 \in \R^n$ such that 
$\lim_{m \to \infty} {h(a_m) \over \| h(a_m) \| } = \alpha$.
Since $h(U)$ has condition $(SSP)$, there is a sequence of points 
$\{ b_m \} \subset U$ tending to $0 \in \R^n$ such that
$$
\| h(a_m) - h(b_m) \| \ll \| h(a_m) \|, \ \| h(b_m) \|. 
$$
It follows that
\begin{equation}\label{seq}
\| a_m - b_m \| \ll \| a_m \|, \ \| b_m \|. 
\end{equation}

On the other hand, there is a subsequence $\{ a_{m_j} \}$
of $\{ a_m \}$ such that
$$
\lim_{{m_j} \to \infty} {a_{m_j} \over \| a_{m_j} \|} 
= \beta \in D(V).
$$
By (\ref{seq}) we have
$$
\lim_{{m_j} \to \infty} {b_{m_j} \over \| b_{m_j} \|} 
= \beta \in D(U).
$$
Since $U$, $V$ are closed cones,
$\beta \in D(U) \cap D(V) \subset U \cap V$.
Let $\tilde{\beta}$ denote the real half line through
$0$ and $\beta$.
Then $\tilde{\beta} \subset U \cap V$.
Note that $\tilde{\beta}$ satisfies condition $(SSP)$.
Therefore there is a sequence of points 
$\{ c_{m_j} \} \subset \tilde{\beta}$ tending to $0 \in \R^n$ such that
$$
\| a_{m_j} - c_{m_j} \| \ll \| a_{m_j} \|, \ \| c_{m_j} \|. 
$$
This implies
$$
\| h(a_{m_j}) - h(c_{m_j}) \| \ll \| h(a_{m_j}) \|, \ \| h(c_{m_j}) \|. 
$$
Thus
$$
\lim_{{m_j} \to \infty} {h(c_{m_j}) \over \| h(c_{m_j}) \|} 
= \lim_{{m_j} \to \infty} {h(a_{m_j}) \over \| h(a_{m_j}) \|} 
= \alpha .
$$
It follows that $\alpha \in D(h(U \cap V))$.
Thus
$D(h(U)) \cap D(h(V)) \subset D(h(U \cap V))$.
\end{proof}

Using a similar argument to the above proposition,
we can generalise it as follows:

\begin{thm}\label{UV}
Let $h : (\R^n,0) \to (\R^n,0)$ be a bi-Lipschitz homeomorphism,
and let $U$, $V \subset \R^n$ be  set-germs at $0\in \R^n$such that 
$0 \in \overline{U} \cap \overline{V}$.
Suppose that $D(U \cap V) = D(U) \cap D(V)$, and 
$U \cap V$ and $h(U)$ satisfy condition $(SSP)$.
Then $D(h(U \cap V)) = D(h(U)) \cap D(h(V))$.
\end{thm}

\begin{rem}\label{remark21}
We cannot drop any assumption from the above theorem.
\begin{enumerate}
\item  $D(U \cap V) = D(U) \cap D(V)$: 
Let $h : (\R,0) \to (\R,0)$ be the identity map,
and let $V = \{ {1 \over m} \ | \ m \in \N \}$ and 
$U = \R \setminus V$.
Then $D(U \cap V) \ne D(U) \cap D(V)$, and
$U \cap V = \emptyset$ and $h(U)= U$ satisfy condition $(SSP)$.
But $D(h(U) \cap h(V)) \ne D(h(U)) \cap D(h((V))$.

\item  $(SSP)$ of $U \cap V$: Let $h : (\R^2,0) \to (\R^2,0)$ be
the inverse of the slow spiral bi-Lipschitz homeomorphism given 
in Example 3.3 of \cite{koikepaunescu} (see Figure 1 and 
Remark \ref{slowspiral}), 
and let $A$, $B$ be spirals on the source space mapped 
by $h$ to two lines $\ell_1$, $\ell_2$ through the origin 
on the target space, respectively.
We set  $U = A \cup (B \setminus m)$ and $V = A \cup (B \cap m )$,
where $m$ is a half line with $0 \in \R^2$ as an end point.
Then $D(U \cap V) = D(U) \cap D(V) = S^1$, $h(U) = \ell_1 
\cup (\ell_2 \setminus C)$, where $C$ is a sequence of points on 
$\ell_2$ convergent to $0 \in \R^2$, satisfies condition $(SSP)$
and $U \cap V = A$ does not satisfy condition $(SSP)$.
On the other hand, we can see that 
$D(h(U) \cap h(V)) = \ell_1 \cap S^1$ and
$D(h(U)) \cap D(h(V)) = (\ell_1 \cup \ell_2) \cap S^1$.

\item  $(SSP)$ of $h(U)$: Let $h : (\R^2,0) \to (\R^2,0)$ be the zigzag
bi-Lipschitz homeomorphism given in Example 3.4 of \cite{koikepaunescu}, 
and let $U = \{ y = 0 \}$ and $V = \{ y = a x \}$ 
for a sufficiently small positive number $a > 0$.
Then $D(U \cap V) = D(U) \cap D(V) = \emptyset$, $U \cap V =
\{ 0 \}$ satisfies condition $(SSP)$ and $h(U)$ does not satisfy 
condition $(SSP)$ (see Remark 5.4 in \cite{koikepaunescu}).
On the other hand,  we can see that $D(h(U \cap V)) = \emptyset$ 
but $D(h(U)) \cap D(h(V)) \ne \emptyset$.  
\end{enumerate}  
\end{rem}

\begin{rem}\label{example1}(Example 5.2 (2) in \cite{koikepaunescu}.)
Let $T$ be an angle with vertex at $O\in \R^2$.
We choose sequences of points $\{ P_m \}$ and $\{ Q_m \}$ 
on the edges of $T$ such that $\overline{OP_m} = {1 \over m^2}$ and 
$\overline{OQ_m} = {1 \over 2}({1 \over m^2} + {1 \over (m+1)^2})$,
and let $C_2$ be the zigzag curve connecting $P_m$'s and $Q_m$'s.
Then $C_2$ satisfies condition $(SSP)$.


 \begin{figure}[htb]
 \centering
 \includegraphics[width=.3\linewidth]{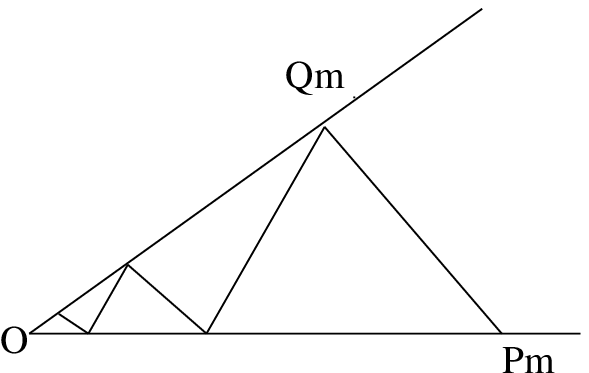}
 \hfil
 \caption{}
 \label{fig:zigzag}
 \end{figure}
Suppose that there are a subanalytic curve $U$ and
a bi-Lipschitz homeomorphism $h : (\R^2,0) \to (\R^2,0)$
such that $h(U) = C_2$ satisfies condition (SSP).
Let $V$ be a half line arbitrarily close to $LD(U)$, therefore close to $U$ as well, such that $U \cap V = \{ 0 \}$ and $D(U \cap V) = D(U) \cap D(V)=\emptyset$. 
On the other hand, $D(h(U \cap V)) = \emptyset$ 
and $D(h(U)) \cap D(h(V)) \neq \emptyset$, as the image $h(V)$ has to be arbitrarily close to the zigzag $C_2$.
By Theorem \ref{UV}, we see 
that $C_2$ cannot be the image of any subanalytic
curve by any bi-Lipschitz homeomorphism.
\end{rem}

\subsection{Directional properties of product sets}

We give some elementary set-theoretical properties 
concerning the  condition $(SSP)$.

\begin{prop}\label{productssp} $(Product)$ 
Let $A \subset \R^m$ be a set-germ at $0\in \R^m$ such that $0\in \overline{A}$ 
and let $B \subset \R^n$ be a set-germ at $0\in \R^n$ such that $0\in \overline{B}$.
Then $A$, $B$ satisfy condition $(SSP)$ at $0 \in \R^m$,
$0 \in \R^n$ respectively if and only if
$A \times B$ satisfies condition $(SSP)$ at 
$(0,0) \in \R^m \times \R^n$.
\end{prop}

\begin{proof} We first show the ``only if'' part.
Let $\{ (a_k,b_k)\}$ be an arbitrary sequence of points of 
$\R^m \times \R^n$ 
tending to $(0,0) \in \R^m \times \R^n$ such that 
$$
\lim_{k \to \infty} {(a_k,b_k) \over \| (a_k,b_k) \|} 
= (a,b) \in D(A \times B).
$$

In the case where $\| a \| , \ \| b \| \ne 0$,
$\lim_{k \to \infty} {a_k \over \| a_k \|} = {a \over \| a \|} 
\in D(A)$ and $\lim_{k \to \infty} {b_k \over \| b_k \|} 
= {b \over \| b \|} \in D(B)$. 
Therefore it is easy to see that there exists a sequence of points
$\{ (c_k,d_k)\}$ of $A \times B$ tending to $(0,0) \in \R^m \times \R^n$ 
such that
$$
\| (a_k,b_k) - (c_k,d_k)\| \ll \| (a_k,b_k) \|, \ \| (c_k,d_k)\| .
$$

Let us assume that $\| a \| = 0$ and $\| b \| = 1$.
Then $\| a_k \| \ll \| b_k \|$  
and $\lim_{k \to \infty} {b_k \over \| b_k \|} \in D(B)$.
Since $B$ satisfies condition $(SSP)$ at $0 \in \R^n$,
there is a sequence of points
$\{ d_k\}$ of $B$ tending to $0 \in \R^n$ such that
$$
\| b_k - d_k\| \ll  \| b_k \|, \ \| d_k\| .
$$
Let $\{ c_j \}$ be a sequence of points of $A$ tending to $0 \in \R^m$ 
such that $\lim_{j \to \infty} {c_j \over \| c_j \|} \in D(A)$.
Take a subsequence $\{ c_{j_k} \}$ of $\{ c_j \}$ so that
$\| c_{j_k} \| < {1 \over k} \| d_k\|$.  
Then $\{ (c_{j_k},d_k)\}$ is a sequence of points of $A \times B$ 
tending to $(0,0) \in \R^m \times \R^n$ such that
$$
\| (a_k,b_k) - (c_{j_k},d_k)\| \ll \| b_k \|, \ \| d_k\| .
$$
It follows that
$$
\| (a_k,b_k) - (c_{j_k},d_k)\| \ll \| (a_k,b_k)\| , \  \| (c_{j_k},d_k)\| .
$$

The case where $\| a \| = 1$ and $\| b \| = 0$ follows 
similarly to the above.
Thus $A \times B$ satisfies condition $(SSP)$
at $(0,0) \in \R^m \times \R^n$.

We next show the ``if'' part.
Since the proof of the other part is the same, it suffices to show 
that $A$ satisfies condition $(SSP)$ at $0 \in \R^m$.
Let $\{ a_k \}$ be an arbitrary sequence of points of $\R^m$ 
tending to $0 \in \R^m$ such that 
$$
\lim_{k \to \infty} {a_k \over \| a_k \|} = a \in D(A).
$$
We take a sequence of points $\{ b_k \}$ of $\R^n$ 
tending to $0 \in \R^n$ such that 
$$
\lim_{k \to \infty} {b_k \over \| b_k \|} = b \in D(B).
$$
Taking a subsequence if necessary, we may assume that
$\| b_k \| \le \| a_k \|$ for any $k \in \N$, and
$$
\lim_{k \to \infty} {(a_k,b_k) \over \| (a_k,b_k) \|} 
= (pa,\sqrt{1-p^2}b) \in D(A \times B)
$$
where $0 < p \le 1$.
Since $A \times B$ satisfies condition $(SSP)$ at 
$(0,0) \in \R^m \times \R^n$,
there is a sequence of points $\{ (c_k,d_k)\}$ of $A \times B$ 
tending to $(0,0) \in \R^m \times \R^n$ such that
$$
\| (a_k,b_k) - (c_k,d_k)\| \ll 
\| (a_k,b_k) \|, \ \| (c_k,d_k) \| .
$$
It follows that
$$
\| a_k - c_k \| \ll \| a_k \|, \ \| c_k \| .
$$
Thus $A$ satisfies condition $(SSP)$ at $0 \in \R^m$.
\end{proof}

\begin{prop}\label{dproduct}
Let $A \subseteq \R^m, B \subseteq \R^n$ be  set-germs at $0\in \R^m$ and 
$0\in \R^n$ respectively, such that 
$0\in \overline{A}, 0\in \overline{B}$.
Then 
$$
D(A\times B)\subseteq \{(ta,\sqrt{1-t^2}b) \ | \ a\in D(A), \ 
b\in D(B), \ t\in [0,1]\}.
$$
Moreover if both $A$ and $B$ satisfy condition $(SSP)$,
then the equality holds.
\end{prop}

\begin{proof}

Let $\{ (a_k,b_k)\}\in A\times B$ be an arbitrary sequence of points   
tending to $(0,0) \in \R^m \times \R^n$ such that 
$$
\lim_{k \to \infty} {(a_k,b_k) \over \| (a_k,b_k) \|} 
= (a,b) \in D(A \times B).
$$

We must have at least one of $\| a \| \ne 0$ or $ \| b \| \ne 0$, hence we get that
$\lim_{k \to \infty} {a_k \over \| a_k \|} = {a \over \| a \|} 
\in D(A)$ or $\lim_{k \to \infty} {b_k \over \| b_k \|} 
= {b \over \| b \|} \in D(B)$. In any case we  take $t=\| a \|=
\sqrt{1-\| b \|^2}$. In the case $a=0$ then $t=0$ and $b\in D(B)$ so we can write 
$(0,b)$ as required.

For the other inclusion, let 
$(ta,\sqrt{1-t^2}b), \, a \in D(A), \,  b \in D(B),$ for some $ t\in [0,1]$. If 
$t\ne 0,1,$ 
then take $s= {\sqrt{1-t^2}\over t}$ and consider a sequence of points
$(t_ia, st_ib),\, t_i\to 0,$ such that $\frac{(t_ia, st_ib)}{\|(t_ia, st_ib)\|}\to
(ta,stb).$
Using the fact that $A$ and $B$ satisfy the condition $(SSP)$ we can find 
$a_i\in A, b_i\in B$ such that $\|a_i-t_ia\|\ll t_i, \|b_i-st_ib\|\ll st_i$ and
this implies that $\frac{(a_i,b_i)}{\|(a_i,b_i)\|}\to (ta,\sqrt{1-t^2}b)$. The 
case when $t=0,1$ is 
trivial.
(We can always reduce the $(SSP)$ 
property to the case when the points are on a line.)

\end{proof}

\subsection{Complex Sequence Selection Property}\label{CSS}

We next consider the complex tangent cone and introduce a complex analogue for the 
condition (SSP).
Let $A \subset \C^n$ be a set-germ at $0 \in \C^n$ such that
$0 \in \overline{A}$.
The {\em complex tangent cone} of $A$ is defined as follows:
$$
LD^*(A) := \biggl\{v \in \C^n \ | \
\begin{matrix}
\exists  \{ c_i \} \subset \C, \
\exists  \{ v_i \} \subset A \setminus \{ 0 \} \to 0 \in \C^n \\
\ \text{s.t.} \ 
\lim_{i \to \infty} c_i v_i = v 
\end{matrix}\biggr\}.
$$
Note that if $A$ is a real (resp. complex) vector space, then 
$LD(A)=A, \, LD^*(A)=A+iA$ 
(resp. $LD^*(A)=A$). 

Let $\C D(A) = \{ c v \in \C^n \ | \ c \in \C, v \in D(A) \}$.
Then we have following:

\begin{lem}\label{complex cone}
$LD^*(A) = \C D(A)$.
\end{lem}

\begin{proof}
Since the inclusion $LD^*(A) \supset \C D(A)$ is obvious, 
we show the converse inclusion.
Note that $0 \in LD^*(A) \cap \C D(A)$.
Take an element $v \in LD^*(A) \setminus \{ 0 \}$.
By definition, there exist  $\{ c_i \} \subset \C$
and $\{ v_i \} \subset A \setminus \{ 0 \} \to 0 \in \C^n$
such that $\lim_{i \to \infty} c_i v_i = v$.
Then there are subsequences $\{ c_{i_j} \}$ of $\{ c_i \}$
and $\{ v_{i_j} \}$ of $\{ v_i \}$
such that $\lim_{j \to \infty} { c_{i_j} \over \| c_{i_j} \| } = c
\in \C \setminus \{ 0 \}$ and
$\lim_{j \to \infty} { v_{i_j} \over \| v_{i_j} \| } = w \in D(A)$.
Then we have $v = c \| v \| w \in \C D(A)$.
It follows that $LD^*(A) \subset \C D(A)$.
\end{proof}

We define the {\em complex projective direction set} $D^*(A) \subset P\C^{n-1}$ 
of $A$ as the quotient set of $LD^*(A) \setminus \{ 0 \}$
by $\C \setminus \{ 0 \}$.   
Then we have

\begin{lem}\label{complex variety}
(Lemma 8.1 in H. Whitney \cite{whitney})
Let $A \subset \C^n$ be an analytic variety such that $0 \in A$.
Then $LD^*(A)$ is also an analytic variety in $\C^n$ and
$D^*(A)$ is a projective variety.
In addition, we have
$$
\dim_{\C} A = \dim_{\C} LD^*(A) = \dim_{\C} D^*(A) + 1.
$$
\end{lem}

The next lemma follows also from Remark 8.2 
and Theorem 11.8 in \cite{whitney}:

\begin{lem}\label{coincidence}
For an analytic variety $0 \in A \subset \C^n$,
$LD^*(A) = LD(A)$.
\end{lem}

\begin{rem}\label{remark0}
Let $S^1 = \{ e^{i \theta} \ | \ \theta \in \R \}$.
For $A \subset \C^n$ such that $0 \in \overline{A}$,
$LD^*(A) = LD(A)$ if and only if $S^1 D(A) = D(A)$.
Note that $D^*(A)$ is the quotient of $D(A)$ by $S^1$.
\end{rem}

One can also consider the sequence selection property over the complex numbers,
which we denote by  $(CSSP)$.

\begin{defn}\label{CSSP}
Let $A \subset \C^n$ be a set-germ at $0 \in \C^n$
such that $0 \in \overline{A}$.
We say that $A$ satisfies {\em condition} $(CSSP)$,
if for any sequence of points $\{ a_m \}$ of $\C^n$
tending to $0 \in \C^n$ such that 
$\lim_{m \to \infty} {a_m \over \| a_m \| } \in LD^*(A)$,
there is a sequence of points $\{ b_m \} \subset A$ such that
$$
 \| a_m - b_m \| \ll \| a_m \|, \ \| b_m \| . 
$$ 
\end{defn}

\begin{rem}
If  $A$ satisfies {\em condition} $(CSSP)$, then $LD^*(A)= LD(A)$, 
and it is clear that it also satisfies {\em condition} $(SSP)$. 
In particular, Lemma \ref{directionrel} implies that 
$D(h(A))= D(h(LD(A)))=D(h(LD^*(A)))$ and  $LD^*(h(A))=LD^*(h(LD^*(A))).$ 
In general it is not true that 
$D(h(A))=D(h(LD^*(A)))$ implies $A$ satisfies {\em condition} $(SSP)$. 
Amongst the examples of sets satisfying {\em condition} $(CSSP)$ 
we mention the complex tangent cones $LD^*(A)$ and  the complex analytic varieties.
\end{rem}
\begin{prop}
Let $A \subset \C^n$ be a set-germ at $0 \in \C^n$ such that 
$0 \in \overline{A}$. 
Then $A$ satisfies {\em condition} $(CSSP)$ if and only if 
$A$ satisfies {\em condition} $(SSP)$ and $S^1D(A)=D(A)$. Consequently if $A$ 
satisfies 
condition $(SSP)$, then  both  $S^1A \,  \text{and} \, \,  \C A$ satisfy condition 
$(CSSP)$.
\end{prop}
\begin{proof}The direct implication is clear from the comments above. For the other 
implication   
let us consider a sequence  $\{ a_m \}$ of $\C^n$
tending to $0 \in \C^n$, such that 
$\lim_{m \to \infty} {a_m \over \| a_m \| } \in LD^*(A)=\C D(A)$. It follows that 
$\lim_{m \to \infty} {a_m \over \| a_m \| }=ca \in S^1D(A)=D(A)$ by assumption. 
Because 
$A$ satisfies {\em condition} $(SSP)$ it follows that there are $b_m\in A$ such that 
$\| a_m - b_m \| \ll \| a_m \|, \ \| b_m \| $, that is $A$ satisfies 
{\em condition} 
$(CSSP)$.

\end{proof}
\bigskip

\section{Transversality.}\label{transversality}
\subsection{Transversality for singular sets}
\label{singular case}

Let us define the notion of transversality for complex analytic
varieties, using the complex tangent cones.

\begin{defn}\label{general transversality}
Let $0 \in A, \ B \subset \C^n$ be analytic varieties.
Then we say that $A$ and $B$ are {\em transverse at}
$0 \in \C^n$ if the following equality holds:
$$
\dim_{\C} LD^*(A) + \dim_{\C} LD^*(B) - n 
= \dim_{\C} (LD^*(A) \cap LD^*(B))
$$
\end{defn}

Concerning this transversality, we have

\begin{thm}\label{invtransversality}
Let $h : (\C^n,0) \to (\C^n,0)$ be a bi-Lipschitz homeomorphism,
and let $0 \in A, \ B, \ h(A), \ h(B) \subset \C^n$ 
be analytic varieties.
Then $A$ and $B$ are transverse at $0 \in \C^n$
if and only if $h(A)$ and $h(B)$ are transverse at $0 \in \C^n$.
\end{thm}

\begin{proof}
We show only the ``only if'' part.
The ``if'' part follows similarly.

By assumption, 
$$
\dim_{\C} LD^*(A) + \dim_{\C} LD^*(B) - n 
= \dim_{\C} (LD^*(A) \cap LD^*(B)).
$$
By Lemma \ref{complex variety}, we see that
$$
\dim_{\C} LD^*(A) = \dim_{\C} LD^*(h(A)), \ \
\dim_{\C} LD^*(B) = \dim_{\C} LD^*(h(B)).
$$
Then, using Lemma \ref{coincidence} and Theorem \ref{main theorem},
we can compute $\dim_{\C} ((LD^*(A) \cap LD^*(B))$ as follows:

\begin{eqnarray*}
2 \dim_{\C} ((LD^*(A) \cap LD^*(B)) &=& \dim_{\R}((LD^*(A) \cap LD^*(B)) \\
&=& \dim_{\R}((LD(A) \cap LD(B)) \\
&=& \dim_{\R}((LD(h(A)) \cap LD(h(B))) \\
&=& \dim_{\R}((LD^*(h(A)) \cap LD^*(h(B))) \\
&=& 2 \dim_{\C} (LD^*(h(A)) \cap LD^*(h(B)))
\end{eqnarray*}
Therefore we have
$$
\dim_{\C} LD^*(h(A)) + \dim_{\C} LD^*(h(B)) - n
= \dim_{\C} (LD^*(h(A)) \cap LD^*(h(B)).
$$
Thus $h(A)$ and $h(B)$ are transverse at $0 \in \C^n$.
\end{proof}

\subsection{Weak transversality}
\label{weak transversality}

When dealing with singular sets in the real set up, we find more convenient to 
use a 
weaker form of transversality, in terms of  real tangent cones. This is 
analogous to 
the use of semi-arcs in Real Algebraic Geometry.

\begin{defn}Let $A$, $B \subset \R^n$ be set-germs at $0 \in \R^n$
such that $0 \in \overline{A} \cap \overline{B}$.
We say that $A$ and $B$ are {\em weakly transverse at}
$0 \in \R^n$ if $D(A) \cap D(B) = \emptyset$ 
(if and only if $LD(A)$ and $B$ are weakly transverse at $0 \in \R^n$).
\end{defn}
Concerning this weak transversality, we have the following:

\begin{lem}\label{subanalytic}
Let $A$, $B$ be two set-germs at $0 \in \R^n$
such that $0 \in \overline{A} \cap \overline{B}$,
and let $h : (\R^n,0) \to (\R^n,0)$ be a bi-Lipschitz homeomorphism.
Suppose that $h(A)$ (or $h(B)$) satisfies condition $(SSP)$.
If $D(A) \cap D(B) = \emptyset$, then 
$D(h(A)) \cap D(h(B)) = \emptyset$. 
\end{lem}

As a corollary of this we have the following.

\begin{thm}\label{invweaktransversality}
Let $A$, $B$ be two set-germs at $0 \in \R^n$
such that $0 \in \overline{A} \cap \overline{B}$,
and let $h : (\R^n,0) \to (\R^n,0)$ be a bi-Lipschitz homeomorphism.
Suppose that $A$ or $B$ satisfies condition $(SSP)$,
and $h(A)$ or $h(B)$ satisfies condition $(SSP)$.
Then $A$ and $B$ are weakly transverse at $0 \in \R^n$ if and only 
if $h(A)$ and $h(B)$ are weakly transverse at $0 \in \R^n$.
\end{thm}

\begin{proof}[Proof of Lemma.]

By hypothesis, $LD(A) \cap LD(B) = \{ 0 \}$.

Assume that $h(A)$ and $h(B)$ are not weakly transverse at $0 \in \R^n$.
Namely, there are a half line 
$\ell \subset LD(h(A)) \cap LD(h(B))$ and 
a sequence of points $\{ b_m \} \subset B$
tending to $0 \in \R^n$ such that
$\lim_{m \to \infty} {h(b_m) \over \| h(b_m) \|} = D(\ell )$.
Here $LD(\ell ) = \ell \subset LD(h(A)) \cap LD(h(B))$.

Since $h(A)$ satisfies condition $(SSP)$,
there is a sequence of points $\{ a_m \} \subset A$ such that
$$
\| h(a_m) - h(b_m) \| \ll \| h(a_m) \|, \ \| h(b_m) \|. 
$$
It follows that
\begin{equation}\label{seqdis}
\| a_m - b_m \| \ll \| a_m \|, \ \| b_m \|. 
\end{equation}
Taking a subsequence of $\{ b_m \}$ if necessary,
we may assume that
$\lim_{m \to \infty} {b_m \over \| b_m \| } = \hat{b} \in D(B)$.
By (\ref{seqdis}), 
$\hat{b} = \lim_{m \to \infty} {a_m \over \| a_m \| } \in D(A)$.
Thus it follows that $D(A) \cap D(B) \neq \emptyset$,
which contradicts the hypothesis.
Thus it follows that  $h(A)$ and $h(B)$ are weakly transverse 
at $0 \in \R^n$.
\end{proof}

\begin{rem}\label{remark5}
We cannot drop the assumption of $(SSP)$ from the above theorem.
For instance, consider Figure 1, the ``slow spiral'' bi-Lipschitz homeomorphism
pictured before. 

\end{rem}

As a corollary of Theorem \ref{invweaktransversality},
we have the following:
\begin{cor}\label{conetransversal}
Let $A$, $B$ be two set-germs at $0 \in \R^n$
such that $0 \in \overline{A} \cap \overline{B}$,
and let $h : (\R^n,0) \to (\R^n,0)$ be a bi-Lipschitz homeomorphism.
Suppose that $h(LD(A))$ satisfies condition $(SSP)$.
Then $A$ and $B$ are weakly transverse at $0 \in \R^n$ 
if and only if $h(LD(A))$  and $h(B)$ are weakly transverse 
at $0 \in \R^n$.

\end{cor}

The following is a simple corollary of Theorem \ref{samedimension}.
\begin{cor}
Let $h:(\R^n,0)\to (\R^n,0)$ be a bi-Lipschitz homeomorphism such that $h$ 
satisfies 
condition semiline-(SSP) and $A, B \subset \R^n$ two arbitrary set-germs at 
$0\in \R^n$  such that 
$0 \in \overline{A} \cap \overline{B}$. Then $A$ and $B$ are weakly transverse at 
$0\in \R^n$ if and only if $h(A)$ and $h(B)$ are weakly transverse at $0\in \R^n$.
\end{cor}



\subsection{Applications to complex analytic varieties}\label{cxvar}

Having developed our transversality theory specifically to  deal with  the singular 
situations, let us apply (illustratively) the above results to arbitrary 
complex analytic 
varieties.
We first give an important proposition.

\begin{prop}\label{drop}
Let $A, \ B \subset \C^n$ be set-germs at $0\in \C^n$such that 
$0 \in \overline{A} \cap \overline{B}$.
If $LD(A) \cap LD^*(B) = \{ 0 \}$,
then $LD^*(A) \cap LD^*(B) = \{ 0 \}$.
\end{prop}

\begin{proof}
Assume that there exists $v \in LD^*(A) \cap LD^*(B)$
such that $v \neq 0 \in \C^n$.
Then, by Lemma \ref{complex cone}, there is a non-zero
$c \in \C$ such that $c v \in LD(A) \cap LD^*(B)$.
This contradicts the hypothesis.
Thus the statement follows.
\end{proof}

As a corollary of Proposition \ref{drop} and Lemma \ref{coincidence},
we have

\begin{cor}\label{equitransverse}
Let $0 \in V \subset \C^n$ be an analytic variety,
and let $A \subset \C^n$ such that $0 \in \overline{A}$.
Then $LD^*(A) \cap LD^*(V) = \{ 0 \}$ if and only if
$LD(A) \cap LD(V) = \{ 0 \}$.
\end{cor}

Let $0 \in V, \ W \subset \C^n$ be analytic varieties,
and let $A$ be a subset of $\C^n$ such that $0 \in \overline{A}$.
Suppose that there exists a bi-Lipschitz homeomorphism
$h : (\C^n,0) \to (\C^n,0)$ such that $h(V) = W$.
Then, by Lemma \ref{subanalytic}, Lemma \ref{coincidence}
and Proposition \ref{drop}, we can see the following:

\begin{thm}\label{transversalanalytic}
$LD^*(A) \cap LD^*(V) = \{ 0 \}$ if and only if
$LD^*(h(A)) \cap LD^*(W) = \{ 0 \}$.
\end{thm}

We consider also the application to the singular points sets 
of complex analytic varieties.
Let $V$, $W$, $A$ and $h$ be the same as above.
Let us denote by $\Sigma (V)$ (resp. $\Sigma (W)$) the singular points
set of $V$ (resp. $W$). 
Note that $h(\Sigma (V)) = \Sigma (W)$.

By Lemma \ref{subanalytic}, we can easily see the following:

\begin{prop}\label{singular set}
$A$ and $\Sigma (V)$ are weakly transverse at $0 \in \C^n$
if and only if $h(A)$ and $\Sigma (W)$ are weakly transverse
at $0 \in \C^n$.
\end{prop}

Let us apply our proposition \ref{closed cone} 
to complex analytic hypersurfaces.
Let $0 \in V, \ W \subset \C^n$ be analytic hypersurfaces,
and let the ideals $I(V)$ and $I(W)$ of $V$ and $W$ be generated
by complex analytic functions $f$ and $g$, respectively.
Let $f_d$ and $g_k$ be the initial homogeneous forms of 
$f$ and $g$, respectively. 

We note that for a hypersurface $V=\{f=0\}$, as above, we have $LD(V)=LD^*(V)=\{f_d=0\}.$
Suppose that there exists a bi-Lipschitz homeomorphism
$h : (\C^n,0) \to (\C^n,0)$ such that $h(V) = W$.
Then, by Lemma \ref{directionrel}, we have

\vspace{3mm}

\noindent {\bf Observation 1.}$LD(h(LD(V) )) = LD(h(V))$. 

\vspace{3mm}

In addition, by Proposition \ref{h(LD(A))}, we have

\vspace{3mm}

\noindent {\bf Observation 2.} $h(LD(V))$ satisfies condition $(SSP)$.

\vspace{3mm}

Using these facts, we can show the following:

\begin{cor}\label{corollary1}
Let $A \subset \C^n$ be a set-germ at $0 \in \C^n$
such that $0 \in \overline{A}$.
Then we have
$$
LD(h(LD(A) \cap LD(V) ))
= LD(h(LD(A))) \cap LD(W) .
$$
\end{cor}

\begin{proof}
By Observation 2, $h(LD(V) )$ satisfies condition $(SSP)$.
Then it follows from Proposition \ref{closed cone} 
and Observation 1 that

\begin{eqnarray*}
LD(h(LD(A) \cap LD(V) )) &=& LD(h(LD(A))) \cap 
LD(h(LD(V))) \\
&=& LD(h(LD(A))) \cap LD(W) .
\end{eqnarray*}

\end{proof}
 We end this section with an application to analytic curves.  

Let $W_1, W_2$ be the 
set-germs of two analytic curves at $0\in \C^n$. 
 
 Then $LD(W_1)=\cup_{j=1}^s m_j, m_i\cap m_j=\{0\}, i\neq j ,\,  
LD(W_2)=\cup_{j=1}^t l_j, l_i\cap\l_j=\{0\}, i\neq j ,\, s,t \in \N$, 
where $m_j, l_j$ are complex lines through $0\in \C^n$.
 \begin{prop}
  Suppose that there is a bi-Lipschitz homeomorphism $h:(\C^n,0)\to(\C^n,0)$ such 
that 
$h(W_1)=W_2$. Then $s=t$.
 \end{prop}
\begin{proof}
We are going to use the known fact that the tangent cone of an irreducible complex 
curve 
is just a complex line.
We know that 
$$\cup_{j=1}^t l_j=LD^*(W_2)=LD(W_2)=LD(h(LD(W_1)))=\cup_{j=1}^s LD^*(h(m_j)).$$
This shows that for any $j, 1 \le j \le s$ , $LD^*(h(m_j))$ consists of some lines 
$l_k$. 
We will show that we cannot have more than one $l_k$. Indeed assume that $l_1, l_2$ 
are 
in $LD^*(h(m_1))$, for convenience.
This would imply that there are sequences $a_i, b_i\in W_1$  realising the 
direction 
$m_1$ so that their images $h(a_i), h(b_i)$ realise $l_1$ and $l_2$ respectively. 
As  
$l_1, l_2$ are distinct directions, following the cited result it follows that the 
sequences $h(a_i)$ and $h(b_i)$ are in different irreducible components of $W_2$, 
say 
in $V_1$ and in $V_2$ respectively. As $h$ is a homeomorphism it follows that 
$a_i\in h^{-1}(V_1)$ and $b_i\in h^{-1}(V_2)$ are also on different irreducible 
components of $W_1$. This contradicts our Theorem \ref{transversalanalytic}. It 
follows 
that  each $LD^*(h(m_j))$ consists  exactly of one line and therefore $s\ge t$ . By 
symmetry we conclude our proof.
\end{proof}
\begin{rem}
It is not difficult to see that the above result does not hold for  $h$ merely a 
homeomorphism.
\end{rem}

\section{(SSP) mappings}
In this section we introduce and investigate the notion of (SSP) mappings.

\begin{defn} Let $A \subset \R^m$ be a set-germ at $0\in \R^m$ such that 
$0\in \overline{A}$ 
and  $B \subset \R^n$ a set-germ at $0\in \R^n$ such that $0\in \overline{B}$. 
Let $h : (A,0) \to (B,0)$ be
 an arbitrary map (or a homeomorphism) germ.
We say that  $h$ is an $(SSP)$ 
{\em map ( $(SSP)$ homeomorphism)} if the graph of $h$ satisfies condition $(SSP)$ 
at $(0,0) \in \R^m \times \R^n$.
\end{defn}

Subanalytic maps and definable maps in an o-minimal structure  are examples of (SSP) 
maps. Also the Cartesian 
product of two (SSP) maps is an (SSP) map. By Theorem 4.19  weak diffeomorphisms are 
also (SSP) homeomorphisms. A function $h : (\R,0) \to (\R.0)$ whose graph is a zigzag 
given
in Example 2.32 is also an (SSP) map.
(Of course, the zigzag should be expanded to the negative part.) 

We next consider the image of a set satisfying
condition $(SSP)$ by an $(SSP)$ map.
Let $\pi : (\R^n ,0) \to (\R^{n-1},0)$ be the projection
on the first $(n - 1)$ coordinates,
and let $A$ be a set-germ at $0 \in \R^n$
such that $0 \in \overline{A}$.
Then the following result holds:

\begin{prop}\label{projection}
Suppose that $\ker \pi$ and $A$ are weakly transverse
at $0 \in \R^n$.
Then we have

(1) $\pi (LD(A)) = LD(\pi (A))$.

(2) If $A$ satisfies condition $(SSP)$,
then so does $\pi (A)$.
\end{prop}

\begin{proof}
(1) For the first  inclusion $\subseteq$,
it suffices to show
$\pi (D(A)) \subseteq LD(\pi (A))$.
Take $a \in D(A)$.
Then there is a sequence of points 
$\{ a_m \} \subset A \setminus \{ 0 \}$
tending to $0 \in \R^n$ such that 
$\lim_{m \to \infty} {a_m \over \| a_m \| } = a$.
By the weak transversality, $a \notin \ker \pi$.
Since $\pi (a_m) \ne 0$ for sufficiently large $m$,
we may assume that $\pi (a_m) \ne 0$ for any $m$.
Then we have
$$
{\pi (a) \over \| \pi (a) \| } =
\lim_{m \to \infty} {\pi (a_m) \over \| \pi (a_m) \|} 
\in D(\pi (A)).
$$
Hence $\pi (a) \in LD(\pi (A))$.

For the second inclusion 
it suffices to show
$D(\pi (A)) \subseteq \pi (LD(A))$.
Take $b \in D(\pi (A))$.
Then there is a sequence of points 
$\{ a_m \} \subset A \setminus \{ 0 \}$
tending to $0 \in \R^n$ such that 
$\lim_{m \to \infty} {\pi (a_m) \over \| \pi (a_m) \| } = b$.
Because of the same reason as above,
we may assume that $\pi (a_m) \ne 0$ for any $m$.
Then, by the weak transversality, there is a subsequence 
$\{ a_{m_j} \}$ of $\{ a_m \}$ such that
$$
\lim_{{m_j} \to \infty} {a_{m_j} \over \| a_{m_j} \|} 
= a \in D(A) \ \ \text{and} \ \ \pi (a) \ne 0.
$$
Then we have
$$
b = 
\lim_{{m_j} \to \infty} {\pi (a_{m_j}) \over \| \pi (a_{m_j}) \|}
= {\pi (a) \over \| \pi (a) \| }
= \pi ({a \over \| \pi (a) \| }) \in \pi (LD(A)).
$$

(2) Let $\{ b_m \}$ be an arbitrary sequence of points of 
$\R^{n-1}$ tending to $0 \in \R^{n-1}$ such that
$$
\lim_{m \to \infty} {b_m \over \| b_m \|} = b \in D(\pi (A)).
$$
Let $\ell = \{ tb \ | \ t \ge 0 \} \subset LD(\pi (A))$.
Then by (1), there is a half line $L \subset LD(A)$ such that 
$\pi (L) = \ell$.
Let us express $L$ as $\{ (t(b,c) \ | \ t \ge 0 \}$ 
for some $c \in \R$.
Let $\alpha_m = (b_m,\| b_m \| c)$ for each $m$.
Then we have
$$
\lim_{m \to \infty} {\alpha_m \over \| \alpha_m \|}
= \lim_{m \to \infty} {(b_m,\| b_m \| c) \over 
\| (b_m,\| b_m \| c) \|}
= \lim_{m \to \infty} {({b_m \over \| b_m \|},c) \over 
\| ({b_m \over \| b_m \|},c) \|}
= {(b,c) \over \| (b,c)\|} \in D(A).
$$
Since $A$ satisfies condition $(SSP)$,
there is a sequence of points 
$\{ \beta_m \} \subset A$, where $\beta_m = (a_m,d_m)
\in \R^{n-1} \times \R$, tending to $0 \in \R^n$ such that
$$
\| \beta_m - \alpha_m \| \ll \| \beta_m \| , \ \| \alpha_m \| .
$$
It follows that
$$
\| \pi (\beta_m ) - \pi (\alpha_m ) \|
= \| \pi (\beta_m - \alpha_m ) \| 
\le \| \beta_m - \alpha_m \| 
\ll \| \beta_m \| .
$$
Then, by the weak transversality,
$$
\| \pi (\beta_m ) - \pi (\alpha_m ) \|
\ll \| \pi (\beta_m ) \| \ \ (\text{and also} 
\ \| \pi (\alpha_m ) \| ).
$$
This means 
$$
\| a_m - b_m \| \ll \| a_m \| , \ \| b_m \| .
$$
Since $a_m = \pi (\beta_m ) \in \pi (A)$,
$\pi (A)$ satisfies condition $(SSP)$.
\end{proof}

\begin{rem}\label{remark31}
We cannot drop the assumption of the weak transversality
in the above theorem.

Let $\pi : \R^3 \to \R^2$ be the projection
defined by $\pi (x,y,z) = (x,y)$, and
let $A = \{ z^4 = x^2 + y^2 \} \cap \pi^{-1}(S)$,
where $S$ is a slow spiral on $(x,y)$-plane.
Then we can see that $A$ satisfies condition $(SSP)$,
but $\pi (A) = S$ does not satisfy condition $(SSP)$.
In addition, $\pi (LD(A)) = \{ 0 \}$ but $LD(\pi (A)) = \R^2$.
\end{rem}

Concerning the weak transversality assumption of 
Proposition \ref{projection}, we have the following lemma.

\begin{lem}\label{WTassumption}
Let $f : (\R^n,0) \to (\R^p,0)$ be a  map such that there is $c>0$ with 
$|f(x)|\le c|x|$
in a neighbourhood of the origin. Let $\pi : \R^n \times \R^p \to \R^n$ be the 
projection
on the first $n$-coordinates.
Then $\ker \pi$ and the graph of $f$ are weakly transverse at
$(0,0) \in \R^n \times \R^p$.
\end {lem}

Using  Proposition \ref{projection} and Lemma \ref{WTassumption},
we can show the following theorem on the $(SSP)$ structure:

\begin{thm}\label{image}
Let $h : (\R^n,0) \to (\R^n,0)$ be a Lipschitz 
homeomorphism such that  $c_1|x|\le |h(x)| \le c_2|x|,$ for some $ c_1, c_2>0$, 
in a 
neighbourhood of $0\in \R^n$, and let $A\subset \R^n$ be a set-germ at $0\in \R^n$ 
such 
that $0 \in \overline{A}$.
Suppose that $A$ satisfies condition $(SSP)$
and $h$ is an $(SSP)$ map. 
Then $h(A)$ also satisfies condition $(SSP)$.
\end{thm}

\begin{proof}
Let $\pi_2 : \R^n \times \R^n \to \R^n$ be the projection
on the second $n$-coordinates, and let
$$
G_A := \{ (a,h(a)) \in \R^n \times \R^n \ | \ a \in A\} .
$$
Suppose that $G_A$ satisfies condition $(SSP)$ as a set-germ
at $(0,0) \in \R^n \times \R^n$.
Since $h^{-1} : (\R^n,0) \to (\R^n,0)$ satisfies $|h^{-1}(x)|\leq \frac{1}{c_1}|x|$ 
in a 
neighbourhood of $0\in \R^n$,
it follows from Proposition  \ref{projection} and Lemma \ref{WTassumption}
that $h(A)= \pi_2(G_A)$ satisfies condition $(SSP)$.
Therefore it suffices to show that $G_A$ satisfies
condition $(SSP)$.

Let $\pi_1 : \R^n \times \R^n \to \R^n$ be the projection
on the first $n$-coordinates, and let $G$ be the graph of $h$.
Since $\ker \pi_1$ and $G$ are weakly transverse at
$(0,0) \in \R^n \times \R^n$, so are $\ker \pi_1$ and $G_A$.

Let us show that $G_A$ satisfies condition $(SSP)$.
Let $\{ \alpha_m \}$ be an arbitrary sequence of points of 
$\R^n \times \R^n$ tending to $(0,0) \in \R^n \times \R^n$ 
such that
$$
\lim_{m \to \infty} {\alpha_m \over \| \alpha_m \|} 
= \alpha \in D(G_A) \subset D(G),
$$
where $\alpha_m = (b_m , c_m ) \in \R^n \times \R^n$
for $m \in \N$.
Let $L = \{ t \alpha \ | \ t \ge 0 \} \subset LD(G_A)$.
Since $G$ satisfies condition $(SSP)$, there is a sequence 
of points of $\{ \beta_m \} \subset G$ such that
\begin{equation}\label{3-1}
\| \alpha_m - \beta_m \| \ll \| \alpha_m \| , \ \| \beta_m \| , 
\end{equation}
where $\beta_m = (d_m , h(d_m )) \in \R^n \times \R^n$
for $m \in \N$.
By the weak transversality of $\ker \pi_1$ and $G_A$,
$\pi_1 (L) = \ell \subset LD(A)$.
Note that $\ell = \{ tb \ | \ t \ge 0 \}$
for $b = \lim_{m \to \infty} {b_m \over \| b_m \|} \in D(A)$.
Therefore it follows from the weak transversality that
\begin{equation}\label{3-2}
\| b_m - d_m \| \ll \| b_m \| , \ \| d_m \| .
\end{equation}

On the other hand, since $A$ satisfies condition $(SSP)$,
there is a sequence of points $\{ a_m \} \subset A$
tending to $0 \in \R^n$ such that
\begin{equation}\label{3-3}
\| a_m - b_m \| \ll \| a_m \| , \ \| b_m \| .
\end{equation}
It follows from (\ref{3-2}) and (\ref{3-3}) that
\begin{equation}\label{3-4}
\| a_m - d_m \| \ll \| a_m \| , \ \| d_m \| .
\end{equation}
Because $h$ is Lipschitz, (\ref{3-4}) implies that,
\begin{equation}
\| h(a_m ) - h(d_m ) \| \ll \|a_m  \| , \ \| d_m  \| .
\end{equation}
Consequently our assumption on $h$ implies that
\begin{equation}\label{3-5}
\| h(a_m ) - h(d_m ) \| \ll \| h(a_m ) \| , \ \| h(d_m ) \| .
\end{equation}

Let $\gamma_m = (a_m , h(a_m )) \in G_A$ for $m \in \N$.
It follows from (\ref{3-4}) and (\ref{3-5}) that
\begin{equation}\label{3-6}
\| \gamma_m - \beta_m \| \ll \| \gamma_m \| , \ \| \beta_m \| .
\end{equation}
By (\ref{3-1}) and (\ref{3-6}) we have
$$
\| \alpha_m - \gamma_m \| \ll \| \alpha_m \| , \ \| \gamma_m \| .
$$
Therefore $G_A$ satisfies condition $(SSP)$.
This completes the proof of Theorem \ref{image}.
\end{proof}

\begin{defn}We call a homeomorphism : $(\R^n,0) \to (\R^n,0)$
an $(SSP)$ {\em bi-Lipschitz homeomorphism}
if it is bi-Lipschitz and an $(SSP)$ map.
\end{defn}
Obviously a  $C^1$ diffeomorphism  $h: (\R^n,0) \to (\R^n,0)$
is an $(SSP)$ bi-Lipschitz homeomorphism.

As a special case of the above theorem  we have the following preserving $(SSP)$ 
structure Theorem.

\begin{thm}\label{image1}
Let $h : (\R^n,0) \to (\R^n,0)$ be an $(SSP)$ bi-Lipschitz 
homeomorphism,  and let $A\subset \R^n$ be a set-germ at $0\in \R^n$ such 
that $0 \in \overline{A}$.
Then  $A$ satisfies condition $(SSP)$
if and only if  $h(A)$  satisfies condition $(SSP)$.
\end{thm}

We have a corollary of the proof of  Theorem \ref{image}.

\begin{cor}\label{graph}
Let $h : (\R^n,0) \to (\R^n,0)$ be a Lipschitz 
homeomorphism as in  Theorem \ref{image}, and let $A\subset \R^n$ be a set-germ 
at $0\in \R^n$
such 
that $0 \in \overline{A}$.
Suppose that $h$ is an $(SSP)$ map and $A$ satisfies
condition $(SSP)$.
Then the restriction $h|_A$ is an $(SSP)$ map.

\end{cor}

We can give a characterisation of an $(SSP)$ map as follows:.

\begin{prop}\label{restrictedSSP}
Let $h : (\R^n,0) \to (\R^n,0)$ be a Lipschitz homeomorphism as in  Theorem 
\ref{image}.
Then $h$ is an $(SSP)$ map if and only if its restrictions to
any semiline $\ell$ are $(SSP)$ maps. 
\end{prop}
\begin{proof} The corollary above gives the necessity. 
Let $G$ be the graph of $h$.
To prove the sufficiency, let us consider a sequence of points 
$\{ (a_m,b_m)\}$ of $\R^n \times \R^n$ tending to
$(0,0) \in \R^n \times \R^n$ such that 
$$
\lim_{m \to \infty} {(a_m,b_m) \over \| (a_m,b_m) \|} 
= (a,b) \in D(G).
$$
We put $l:=\{(ta,tb) \ | \ t \geq 0 \}$ and
$l_1:=\{ta \ | \  t\geq 0 \}$.
Then there is a sequence of points $\{ (c_i,h(c_i))\}$ 
of $G$ such that
$\lim_{i \to \infty} {(c_i,h(c_i)) \over \| (c_i,h(c_i)) \|} 
= (a,b)$. 
Since $l$ satisfies condition $(SSP)$, there are positive numbers 
$s_i\in \R$ so that
$$
\|s_i(a,b)-(c_i,h(c_i))\| \ll \| c_i\|, \ s_i.
$$ 
This shows that the direction $l$ is also attained by the sequence 
$\{ (s_ia, h(s_ia))\}$, namely it appears as a direction of the graph 
of the restriction of $h$ to $l_1$, 
and we can apply the hypothesis to end the proof.
\end{proof}
\begin{rem}
Unfortunately a homeomorphism which is merely an (SSP) homeomorphism, does not 
always 
preserve the condition (SSP). We can construct an (SSP) homeomorphism 
$h:\R \to \R$, 
which also satisfies semiline-$(SSP)$, 
such that there is a set $A$ satisfying condition (SSP) but $h(A)$ does not.
\end{rem}
Concerning Theorem \ref{image1}, it may be natural to ask
the following question:

\begin{ques}\label{question1}
Let $h : (\R^n,0) \to (\R^n,0)$ be a bi-Lipschitz homeomorphism.
Suppose that if $A$ satisfies condition $(SSP)$,
so does $h(A)$ for any set-germ $A$ at $0 \in \R^n$ 
such that $0 \in \overline{A}$.
Then is $h$ an $(SSP)$ map?
\end{ques}
We have a negative example to the above question.

\begin{example}\label{example31}
Let $h : (\R,0) \to (\R,0)$ be a zig-zag function
whose graph is drawn below (Figure 3). (Note that the zigzag in Figure 2 is not the graph of a function!)

\begin{figure}[ht]
  \centering
  \begin{tikzpicture}[scale=1]
   
   \draw[->] (0,-4) -- (0,7);
   \draw[->] (-4,0) -- (7,0);
   
  
\draw[dotted] (0,0) -- (7,7);
\draw[dotted] (0,0) -- (6,3);
  \draw (7,7) -- (6,3) ; 
\draw (6,3) -- (2,2) ;
    \draw (2,2) -- (1,0.5);
   \draw (1,0.5) -- (0.25, 0.25) ;
\draw (0.25,0.25) -- (0.1,0.05);
\draw (0,0) -- (-3,-3);
   
  \end{tikzpicture}
  \caption{}
  \label{fig.3}
\end{figure}

\noindent Then $h$ is a bi-Lipschitz homeomorphism.
As stated in Remark \ref{remark1} (2), $h$ satisfies the 
$(SSP)$ assumption in Question \ref{question1}.
But the graph of $h$ does not satisfy condition $(SSP)$.
Therefore $h$ is not an $(SSP)$ map, moreover $h$ also satisfies condition 
semiline-$(SSP)$.
\end{example}

\begin{rem}\label{remark32}
We can consider a similar question to Question \ref{question1}
in the semialgebraic category or in the subanalytic one.
Namely, we consider the question, replacing condition $(SSP)$
with semialgebraic or subanalytic.
Indeed, let $h : (\R^n,0) \to (\R^n,0)$ be a bi-Lipschitz homeomorphism.
Suppose that if $A$ is semialgebraic (or subanalytic), then
so does $h(A)$ (for any set-germ $A$ at $0 \in \R^n$ 
such that $0 \in \overline{A}$).
Does this property imply that $h$ is a semialgebraic map (subanalytic respectively)?
The above example provides  a negative answer.
\end{rem}

This kind of phenomenon is not particular to the one-dimensional case. For instance, 
let $T : = \{ (x,y) \in \R^2 ; x > 0 \, \, \text{and}\, \,  exp(-1/x^2) < y < 2 exp(-1/x^2) \}$.
Then we can define a homeomorphism germ $h : (\R^2,0) \to (\R^2,0)$
by  identity outside $T$, and,  on $T$, we can take any extension so that
$h$ is a non-semialgebraic homeomorphism. 
However this kind of $h$ takes semialgebraic set-germs to semialgebraic set-germs.
Indeed,  for any 1-dimensional semialgebraic set $A$ such that $0 \in \overline A,\,\,
A \cap T$ is empty as a set-germ at $0 \in \R^2$, and obviously
its image $h(A) = A$ is semialgebraic.
If $B$ is an arbitrary 2-dimensional semialgebraic set
such that $0 \in \overline B$,
then the boundary of $B$ does not intersect $T$ as set-germs
at $0 \in \R^2$.
Therefore we can see that $h(B)$ is also a semialgebraic set-germ.

The subanalytic case is similar. 

Concerning the above phenomenon we mention the following results.

\begin{prop}\label{productmap} 
\noindent
\begin{enumerate}
\item Both $\, h_i : (\R^{n_i},0) \to (\R^{n_i},0)$, $i=1,2$, 
are $(SSP)$ bi-Lipschitz  homeomorphisms if and only if 
$h_1\times h_2 : (\R^{n_1} \times \R^{n_2}, 0 \times 0) 
\to (\R^{n_1} \times \R^{n_2}, 0 \times 0)$ is 
an $(SSP)$ bi-Lipschitz homeomorphism.

\item Let $h: (\R^n,0) \to (\R^n,0)$ be a bi-Lipschitz homeomorphism.
Then $I_n\times h : (\R^n \times \R^n, 0 \times 0) \to
(\R^n \times \R^n, 0 \times 0)$  (or  $I_n\times h^{-1}$)
satisfies  condition semiline-$(SSP)$ if and only if 
$I_n\times h : (\R^n \times \R^n, 0 \times 0) \to(\R^n \times \R^n, 0 \times 0)$ 
is an $(SSP)$ map.

\item Let $h : (\R^n,0) \to (\R^n,0)$ be a bi-Lipschitz homeomorphism.
Then $h$ is an $(SSP)$ map if and only if $I_n \times h : 
(\R^n \times \R^n, 0 \times 0) \to (\R^n \times \R^n, 0 \times 0)$ 
( or $I_n\times h^{-1} $)  satisfies condition semiline-$(SSP)$.

\end{enumerate}

(Here $I_n : (\R^n,0) \to (\R^n,0)$ represents the identity map.)
\end{prop}

\begin{proof}
Note that the graph of $h_1\times h_2$ is the Cartesian product 
of the graphs of $h_1$ and $h_2$.
Then (1) follows from Proposition \ref{productssp}.

In (2) we already know the sufficiency by Theorem \ref{image}. 
For necessity, in our set up, it follows that $I_n\times h$ takes 
(SSP) sets to (SSP) sets, see Corollary \ref{linecor2}. 
In particular the diagonal in $\R^n \times \R^n$ is taken to 
the graph of $h$, so $h$ is an (SSP) map and by (1) 
so is $I_n\times h$.

Now (3) clearly  follows from (1) and (2). 
\end{proof}

\begin{rem}\label{remark33}
Note that if $h : (\R^n,0) \to (\R^n,0)$ is an $(SSP)$ bi-Lipschitz homeomorphism, 
then 
for any  semiline $\ell$ the cone 
$LD(G_{\ell})$ is also a semiline. This fact also  explains  the example 
\ref{example31}. (Here $G_{\ell}$ is the graph of the restriction of $h$ to $\ell$.)
\end{rem}

\begin{rem}\label{remark34}
\noindent
\begin{enumerate}
\item There are bi-Lipschitz homeomorphisms
$h : (\R^n,0) \to (\R^n,0)$, $n \ge 2$,
which are not  $(SSP)$ bi-Lipschitz homeomorphisms.

For instance, let $h : (\R^2,0) \to (\R^2,0)$ 
be a zigzag bi-Lipschitz homeomorphism in Example 3.4
of \cite{koikepaunescu}
or a slow spiral bi-Lipschitz homeomorphism,
and let $A$ be the positive $x$-axis.
Clearly $A$ satisfies condition $(SSP)$
and $h(A)$ does not satisfy condition $(SSP)$. 
Then, by Theorem \ref{image1}, $h$ is not an $(SSP)$ map.

\item  The homeomorphism $\overline h$ associated to a bi-Lipschitz  
homeomorphism which satisfies condition semiline-$(SSP)$ is an (SSP) map.
\end{enumerate}
\end{rem}

In order to give another large class of examples of (SSP) homeomorphisms.
let us consider  a category of homeomorphisms $h : (\R^n,0) \to (\R^n,0)$
 called  {\em weak diffeomorphisms}, namely those  $h$ and $h^{-1}$ which admit
derivative (= linear approximation) at $0 \in \R^n$.

We will point out some directional and $(SSP)$ properties 
for  the class of weak diffeomorphisms, namely we will show that the  weak 
diffeomorphisms 
are also (SSP) homeomorphisms.
\begin{rem}
Note that a weak homeomorphism is not necessarily Lipschitz.  
For  instance we may have $h(x,y,z)=(x,y,z+(x^5+y^5)^{1/3}).$ 
\end{rem}
Let $h : (\R^n,0) \to (\R^n,0)$ denote a weak diffeomorphism.
Then $h$ can be expressed in a neighbourhood of $0 \in \R^n$
as follows:
$$h(x) = M_h (x) + O_h (x),$$ 
 where $M_h$ is a regular linear map from $\R^n$ to $\R^n$,
and $\lim_{x \to 0} {\| O_h (x)\| \over \| x\|} = 0$.
Note that $M_{h^{-1}} \circ M_h = Id$. 

\begin{lem}\label{dpforwd} 
Let $A \subset \R^n$ be a set-germ at $0 \in \R^n$
such that $0 \in \overline{A}$,
and let $G$ and $G_A$ be the graphs of the weak diffeomorphism  $h$ and $h|_A$ respectively.
Then we have

(1) $LD(M_h(A)) = M_h(LD(A)) = LD(h(A))$.

(2) $LD(G_A) =LD(\text{graph} (M_h|_A))$. 
In particular  $LD(G)=\text{graph} (M_h)$ is an $n$-dimensional linear subspace
of $\R^n \times \R^n$.
\end{lem}

\begin{proof}
(1) Since we can easily see the first equality,
we show only the second one. Moreover, interchanging $h$ and $h^{-1}$,
it suffices to show
$M_h(LD(A)) \subset LD(h(A))$. 

Let $\alpha$ be an arbitrary element of $D(A)$.
Then there is a sequence of points $\{ a_m \} \subset A$
tending to $0 \in \R^n$ such that 
$\lim_{m \to \infty} {a_m \over \| a_m \|} = \alpha$. 
Therefore

\begin{eqnarray*}
{ M_h(\alpha ) \over \| M_h(\alpha )\|} 
&=& \lim_{m \to \infty} {M_h({a_m \over \| a_m\|}) \over
\| M_h({a_m \over \| a_m\|}) \|}  \\
&=& \lim_{m \to \infty} {{1 \over \| a_m \|}(M_h(a_m) + O_h(a_m)) 
\over {1 \over \| a_m \|} \| M_h(a_m) + O_h(a_m) \|} \\
&=& \lim_{m \to \infty} {M_h(a_m) + O_h(a_m) \over 
\| M_h(a_m) + O_h(a_m) \|} \\
&=& \lim_{m \to \infty} {h(a_m) \over \| h(a_m) \|}
\in D(h(A)).
\end{eqnarray*}

\noindent It follows that $M_h(LD(A)) \subset LD(h(A))$. 

(2) The proof is similar to the above and it is omitted.  
\end{proof}
\begin{rem}

It is also worth mentioning that there are $(SSP)$ homeomorphisms which do not 
satisfy condition semiline-$(SSP)$. For example one may consider the function  
$f$ which has a zig-zag graph  
and the associated homeomorphism $h:(\R^2,0) \to (\R^2,0), \, h(x,y)=(x,y+f(x))$.
This shows that outside the bi-Lipschitz category there is no direct implication 
between the $(SSP)$ homeomorphisms and those satisfying condition semiline-$(SSP)$ 
(see also \ref {example31}).
\end{rem}
 
The following theorem shows that the weak diffeomorphisms are also suitable for the 
$(SSP)$ category.
\begin{thm}\label{weekdiffeo}
A weak diffeomorphism is an $(SSP)$  homeomorphism and satisfies condition 
semiline-$(SSP)$ as well.
\end{thm}

\begin{proof}
Let $h$ be a weak diffeomorphism.
In fact it is an easy consequence of Lemma \ref{dpforwd}  that for any 
$A\subset \R^n$ 
satisfying condition $(SSP)$, $G_A$ satisfies condition $(SSP)$,
where $G_A$ is the graph of the restriction of $h$ to $A$.
Therefore $G$ satisfies condition $(SSP)$ at 
$0 \in \R^{2n}$.
\end{proof}

 As a corollary of the proof above and Lemma \ref{WTassumption} we have the 
following 
corollary.
 
 \begin{cor}
 Let $h:(\R^n,0)\to(\R^n,0)$ be a  weak diffeomorphism and let $A\subset \R^n$ be 
a set-germ at $0\in \R^n$ such that $0\in \overline A$. Then $A$ satisfies 
condition (SSP) if and only if  $h(A)$ satisfies 
condition (SSP).
 \end{cor}

\section{Appendix - Geometric applications to spirals}\label{spirals}

We consider polar coordinates $(r,\theta), 0<r, \theta<\infty$. Let $R:(0,\infty) \to (0,\infty)$ be a continuous function. We say that $S_0: r=R(\theta)$ is a spiral at $0\in \R^2$ if $R$ is strictly monotone and 
$$\lim_{\theta \to \infty} R(\theta)=0 \, \,  \text{or} \lim_{\theta \to \infty} R(\theta)=\infty.$$
In the first case we write $R(\infty)=0$ and note that the extension $R:(0,\infty] \to [0,\infty)$ is continuous and injective. In the second case we write $R(0)=0$ and note that also the extension $R:[0,\infty) \to [0,\infty)$ is continuous and injective.

Let us introduce the homeomorphism germ induced by a spiral, defined in polar coordinates by:
$$h_{S_0}:(\R^2,0)\to(\R^2,0), h_{S_0}(r,\alpha)=(r,R^{-1}(r)+\alpha), \, 0\leq\alpha<2\pi.$$
For $0\leq\alpha<2\pi$ we put  $L_\alpha:=\{(r,\alpha)|0\leq r<\infty\}$ and $S_\alpha:=h_{S_0}(L_\alpha).$
Note that $S_0$ is just the spiral $r=R(\theta)$ together with $0\in \R^2.$

If $0\leq \alpha< 2\pi$ we denote by $R_\alpha$ the rotation of $\R^2$ centred at the origin and of angle $\alpha.$ Then the following is obvious:
\begin{rem}
$S_\alpha=R_\alpha (S_0) \, \, \text{and} \, \, D(h_{S_0}(L_\alpha))=R_\alpha(D(S_0)).$
\end{rem}
For applications to spirals we need Proposition  \ref{closed cone} modified to the following:

\begin{prop}\label{closed cone'}
Let $h : (\R^n,0) \to (\R^n,0)$ be a  homeomorphism,
and let $U$, $V \subset \R^n$ be set-germs  $0 \in \R^n$ such that $0\in \overline U\cap \overline V.$
Suppose that 
\begin{enumerate}
\item $D(U\cap V)=D(U)\cap D(V)$,
\item $U\cap V$ satisfies condition (SSP),
\item $h(U)$ satisfies condition $(SSP)$, and
\item $h$ is bi-Lipschitz.
\end{enumerate}
Then $D(h(U \cap V)) = D(h(U)) \cap D(h(V))$.
\end{prop}

We will use the above proposition to give a classification of spirals. (We note that in general is quite tedious to test the property of being bi-Lipschitz.)
Firstly note that for $\alpha \neq \beta, \, \alpha, \beta \in [0,2\pi)$ we have $D(L_\alpha \cap L_\beta)=D(L_\alpha)\cap D(L_\beta)=\emptyset$
and $L_\alpha \cap L_\beta=\{0\}$ satisfies condition (SSP). That is, $L_\alpha \,\,\text{and} \, L_\beta$ satisfy the first two conditions in Proposition \ref{closed cone'}.

We first consider the case when $\#(D(S_0))>1$, which is equivalent with the following condition:
There are $\alpha\neq \beta , \alpha, \beta \in [0,2\pi)$ such that $D(h_{S_0}(L_\alpha)) \cap D(h_{S_0}(L_\beta))\neq \emptyset .$

On the other hand, as $D(h_{S_0}(L_\alpha\cap L_\beta))=\emptyset,$ the conclusion of Proposition \ref{closed cone'} does not hold;  this may happen only if one or both of conditions $(3)$ and $(4)$ fail. We can therefore divide the case $\#(D(S_0))>1$ in three classes as follows:

\begin{enumerate}[(A)]
\item $S_0$ satisfies condition (SSP) at $0\in \R^2$. In this case the induced homeomorphism $h_{S_0}$ is not bi-Lipschitz. For example, this is the case for the hyperbolic spiral, $r=a/\theta, a>0.$ Note that the length of the spiral is infinite (even for $c\leq \theta <\infty, c>0).$ On the other hand the spiral $r=a/\theta^2$ also satisfies (SSP) so, although its length is finite for $c\leq \theta <\infty, c>0$,  the induced homeomorphism is,  yet again, not bi-Lipschitz.

\item The induced homeomorphism $h_{S_0}$ is bi-Lipschitz, and therefore $S_0$ does not satisfy (SSP). This is the case for the logarithmic spiral $r=ae^{-b\theta}, a, b >0.$

\item In this case $S_0$ does not satisfy (SSP) and $h_{S_0}$ is not bi-Lipschitz.
\end{enumerate}

Finally we have the remaining case when $\#D(S_0)=1.$
This condition is equivalent with the condition 
$\forall \alpha\neq \beta , \alpha, \beta \in [0,2\pi)$ we have  $D(h_{S_0}(L_\alpha)) \cap D(h_{S_0}(L_\beta)) = \emptyset .$ This is the case for the Archimedean spiral $r=a\theta, a>0$.

\medskip
 Let us recall some examples on (SSP) analysed in \cite{koikepaunescu}.

\begin{example}
(1) Let $0 < r < 1$, and let $A = \{a_m| m\in \N\}$ be a sequence
of points of $\R,$ defined by $a_m := r^m$.
Then $A$ does not satisfy (SSP) at $0 \in \R$.

(2) Let $B = \{b_m| m\in \N\}$ and $C = \{c_m| m\in \N\}$ be sequences of points of $\R$
defined by $b_m := 1/m$ and $c_m := (1/m)^2$, respectively.
Then $B$ and $C$ satisfy (SSP) at $0 \in \R$.
\end{example}
The first example  above can be used to construct  spirals belonging to the class (C), whilst the
second one  can be used to explain the examples given in the class (A).









\end{document}